\documentclass[final,envcountsect, envcountsame, runningheads]{llncs}
\pdfoutput=1

\usepackage{etex,etoolbox}
\usepackage[utf8]{inputenc}
\usepackage{amsmath}
\usepackage{amssymb}
\usepackage{makeidx}  % allows for indexgeneration
\usepackage{xspace}  % allows for indexgeneration
\usepackage{ifdraft}
\usepackage{environ}
\usepackage{hyperref}
\usepackage{mathtools}
\usepackage{wrapfig}
\usepackage{nicefrac}
\DeclareMathAlphabet{\mathpzc}{OT1}{pzc}{m}{normal} % for \mathpzc, especially \H

\usepackage{ifthen}
\usepackage{times}
\usepackage{microtype}

\usepackage{enumitem}
\usepackage{dsfont}

\newboolean{proofappendix} \setboolean{proofappendix}{true}

\newboolean{full} \setboolean{full}{true}

\newboolean{shownotes} \setboolean{shownotes}{false}

\iffull\else\setboolean{proofappendix}{true}\fi

\newboolean{trimmedmargins} \setboolean{trimmedmargins}{false}
\usepackage{substr}
\IfSubStringInString{\detokenize{trimmed}}{\jobname}{
\setboolean{trimmedmargins}{true}
}{}
\setboolean{trimmedmargins}{true}

\iftrimmedmargins
\usepackage[marginparwidth=2cm,nohead,nofoot,
            headsep = 1cm,
            footskip = 1cm,
            text={12.2cm,19.3cm},
            paperwidth=16.2cm,
            paperheight=23.3cm,
            marginparsep=1mm,
            marginparwidth=1.8cm,
            footskip=1cm,
            centering
            ]{geometry}
\else
\fi
\usepackage[utf8]{inputenc}
\usepackage[english]{babel}
\usepackage{tikz-cd}
\usepackage{ifthen}
\usepackage[notref,notcite]{showkeys}

\newboolean{usenatbib}
\setboolean{usenatbib}{true}

\ifusenatbib
    \usepackage[numbers]{natbib}

\else
\fi

\ifshownotes
    \newenvironment{notes}{\it}{\upshape}
\else
    \NewEnviron{notes}{}
\fi

\hypersetup{
    colorlinks=true,
    linkcolor=black,
    citecolor=black,
    filecolor=black,
    urlcolor=black,
}

\newenvironment{nicearray}[1]
    {\begin{array}{@{}#1@{}}\hline\noalign{\smallskip}}
    {\\\hline\end{array}}
\newcommand{\spmidrule}{\noalign{\smallskip}\hline\noalign{\smallskip}}

\newcommand{\partitle}[1]{\textbf{#1}}

\spnewtheorem{assumption}[definition]{Assumption}{\bfseries}{\it}

\usepackage[nomargin,marginclue,footnote,status=draft]{fixme}
\fxusetargetlayout{color}
\FXRegisterAuthor{sm}{asm}{SM}%Stefan
\FXRegisterAuthor{tw}{atw}{TW}%Thorsten

\newcommand{\twnotei}[1]{\twnote[nofootnote,nomargin,inline]{#1}}

\newcommand{\textqt}[1]{``#1''}
\newcommand{\op}[1]{\operatorname{\textsf{\upshape {#1}}}}
\newcommand{\B}{\ensuremath{\mathcal{B}}\xspace}
\newcommand{\C}{\ensuremath{\mathcal{C}}\xspace}
\newcommand{\D}{\ensuremath{\mathcal{D}}\xspace}
\newcommand{\Coalg}{\ensuremath{\mathsf{Coalg}}\xspace}
\newcommand{\Coalgfg}{\ensuremath{\mathsf{Coalg}_{\mathsf{fg}}}\xspace}
\newcommand{\Coalglfg}{\ensuremath{\mathsf{Coalg}_{\mathsf{lfg}}}\xspace}
\newcommand{\Coalgfp}{\ensuremath{\mathsf{Coalg}_{\mathsf{fp}}}\xspace}

\newcommand{\Set}{\ensuremath{\mathsf{Set}}\xspace}
\renewcommand{\Vec}{\ensuremath{\mathsf{Vec}}}

\renewcommand{\H}{\ensuremath{\mathpzc{H}}\xspace}
\newcommand{\Hf}{\ensuremath{\mathpzc{H}_\textnormal{\sffamily f}}\xspace}
\newcommand{\Funf}{\ensuremath{\mathsf{Fun}_\textnormal{\sffamily f}}\xspace}
\newcommand{\Mndc}{\ensuremath{\mathsf{Mnd}_\textnormal{\sffamily c}}\xspace}
\newcommand{\Mndf}{\ensuremath{\mathsf{Mnd}_\textnormal{\sffamily f}}\xspace}
\newcommand{\EQ}[1][]{\ensuremath{\mathsf{EQ}_\textnormal{#1}}\xspace}
\newcommand{\id}{\ensuremath{\textnormal{id}}\xspace}
\newcommand{\inj}{\ensuremath{\mathsf{in}}}
\newcommand{\inl}{\ensuremath{\mathsf{inl}}}
\newcommand{\inr}{\ensuremath{\mathsf{inr}}}
\newcommand{\Id}{\ensuremath{\textnormal{Id}}\xspace}
\newcommand{\Pot}{\ensuremath{{\mathcal{P}}}\xspace}
\newcommand{\Potf}{\ensuremath{{\mathcal{P}_\textnormal{\sffamily f}}}\xspace}

\newcommand{\lff}{\ensuremath{\ell}}

\newcommand{\FPS}[2][S]{\ensuremath{#1\llangle#2\rrangle}}
\newcommand{\Poly}[2][S]{\ensuremath{#1\langle#2\rangle}}
\newcommand{\colim}{\ensuremath{\operatorname{colim}}\xspace}
\renewcommand{\o}{\ensuremath{\cdot}}
\newcommand{\fgiterative}{fg-iterative\xspace}
\newcommand{\LFF}{\vartheta}
\newcommand{\fpair}[1]{\ensuremath{\langle #1 \rangle}}
\newcommand{\fuse}[1]{\ensuremath{[#1]}}
\newcommand{\N}{\ensuremath{\mathbb{N}}}
\newcommand{\fp}{{\textsf{\upshape{fp}}}}
\newcommand{\lfp}{{\textsf{\upshape{lfp}}}}
\renewcommand{\Im}{\op{Im}}

\makeatletter
\newsavebox{\@brx}
\newcommand{\llangle}[1][]{\savebox{\@brx}{\(\m@th{#1\langle}\)}%
  \mathopen{\copy\@brx\kern-0.5\wd\@brx\usebox{\@brx}}}
\newcommand{\rrangle}[1][]{\savebox{\@brx}{\(\m@th{#1\rangle}\)}%
  \mathclose{\copy\@brx\kern-0.5\wd\@brx\usebox{\@brx}}}
\makeatother

\usepackage{newunicodechar}
\newunicodechar{×}{\ensuremath{\times}}
\newunicodechar{≤}{\ensuremath{\le}}
\newunicodechar{→}{\ensuremath{\to}}
\newunicodechar{⊗}{\ensuremath{\otimes}}
\newunicodechar{〈}{\ensuremath{\langle}}
\newunicodechar{〉}{\ensuremath{\rangle}}

\tikzset{
    every diagram/.style={
        row sep=1cm,
        column sep=1cm,
    }
}
\newcommand{\descto}[3][]{
    \arrow[draw=none]{#2}[description,#1]{#3}
}
\tikzset{oldequal/.style={
    double equal sign distance,
    -,
}}
\tikzset{shiftarr/.style={
        rounded corners,%
        to path={--([#1]\tikztostart.center)
                     -- ([#1]\tikztotarget.center) \tikztonodes
                     -- (\tikztotarget)},
}}
\tikzset{commutative diagrams/diagrams={
    rounded corners,
}}
\tikzstyle{mathnodes}=[
  nodes={%
   execute at begin node=\(,%
   execute at end node=\)%
  }%
]
\tikzstyle{anchorcenter}=[
    baseline={([yshift=-.5ex]current bounding box.center)}
]
\tikzstyle{lambdatree}=[
    level distance = 4mm,
    anchorcenter,
    level 1/.style={sibling distance=13mm},
    every node/.style={
        inner sep=1pt,
    },
    edge from parent/.style= {
        draw,
        edge from parent path={(\tikzparentnode) -- (\tikzchildnode.north)},
    },
    noedge/.style={
        edge from parent/.style={}
    },
]
\tikzstyle{level}=[
    sibling distance=20mm/#1,
]

\newcommand{\PMNC}{\pgfmatrixnextcell}

\newcommand{\adjointpair}{{%
\hspace{-3mm}\begin{tikzcd}[column sep=8mm,shorten <= 1mm, shorten >=1mm]
        {}
            \arrow[yshift=0.4mm,bend left=16]{r}[name=LeftAdjoint,above]{}
        \pgfmatrixnextcell {}
            \arrow[yshift=-0.4mm,bend left=16]{l}[name=RightAdjoint,below]{}
            \arrow[draw=none,to path =(LeftAdjoint) -- (RightAdjoint) \tikztonodes]{}[anchor=center,sloped]{\dashv}
\end{tikzcd}\hspace{-3mm}}}

\makeatletter
\newbox\xrat@below
\newbox\xrat@above
\newcommand{\xrightarrowtail}[2][]{%
  \setbox\xrat@below=\hbox{\ensuremath{\scriptstyle #1}}%
  \setbox\xrat@above=\hbox{\ensuremath{\scriptstyle #2}}%
  \pgfmathsetlengthmacro{\xrat@len}{max(\wd\xrat@below,\wd\xrat@above)+.6em}%
  \mathrel{\tikz [>->,baseline=-.75ex]
                 \draw (0,0) -- node[below=-2pt] {\box\xrat@below}
                                node[above=-2pt] {\box\xrat@above}
                       (\xrat@len,0) ;}}
\makeatother

\newcommand{\densesubsection}[1]{
\addtocounter{subsection}{1}
\medskip\noindent
{\bf \arabic{section}.\arabic{subsection}\hspace{1em}#1\\[1.5mm]}%
}

\makeatletter
\providecommand{\@fourthoffour}[4]{#4}
\def\fixstatement#1{%
  \AtEndEnvironment{#1}{%
    \xdef\pat@label{\csname #1name\endcsname~\@currentlabel}}}

\globtoksblk\prooftoks{1000}
\newcounter{proofcount}

\NewEnviron{proofatend}{%
  \edef\next{%
     \noexpand\subsection*{Proof of \pat@label}
     \unexpanded\expandafter{\BODY}}%
   \global\toks\numexpr\prooftoks+\value{proofcount}\relax=\expandafter{\next%
   }
   \stepcounter{proofcount}}

\def\printproofs{%
  \count@=\z@
  \loop
    \the\toks\numexpr\prooftoks+\count@\relax
    \ifnum\count@<\value{proofcount}%
    \advance\count@\@ne
  \repeat}
\makeatother

\iffull
\else
\fi

\ifproofappendix
\else
    
\fi
\fixstatement{thm}
\fixstatement{lemma}
\fixstatement{corollary}
\fixstatement{proposition}

\newcommand{\takeout}[1]{\empty}
\newcommand{\Sd}{S}

\titlerunning{The Locally Finite Fixpoint and its Properties}
\authorrunning{S.~Milius, D.~Pattinson and T.~Wißmann}

\title{A New Foundation for Finitary Corecursion}
\subtitle{The Locally Finite Fixpoint and its Properties}
\author{
   Stefan Milius\inst{1}\fnmsep\thanks{Supported by Deutsche Forschungsgemeinschaft (DFG) under project MI~717/5-1}
   \and Dirk Pattinson\inst{2}
   \and Thorsten Wißmann\inst{1}\fnmsep${}^\star$
   }
\institute{Friedrich-Alexander-Universität Erlangen-Nürnberg
           \and The Australian National University}

\begin{document}
\maketitle
\begin{abstract}
\takeout{ % old abstract
We present a new approach to capture finitary behaviours of coalgebras based on
a new notion of local finiteness using finitely generated objects on locally
finitely presentable categories. These finite behaviours are captured by
fixpoint of the functor, we call it the \emph{locally finite fixpoint} (LFF),
which is a subcoalgebra of the final coalgebra. The LFF is characterized by
multiple universal properties: 1.~as the final \emph{locally finitely generated}
coalgebra; 2.~as the initial iterative algebra for equation systems.

With this new approach we can describe 1.~the context-free languages on $\Sigma$ by
a universal property, namely as the LFF of a lifting of $2×(-)^\Sigma$ to the category of $\Sigma$-pointed
idempotent semi-rings; 2.~the monad of Courcelle's algebraic trees by a
universal property, i.e.~as the LFF for an appropriate functor on the category
of pointed finitary monads.}% end takeout
\takeout{
% New abstract
This paper contributes to a uniform theory of the behaviour of ``finite-state''
systems. We propose that such systems are modeled as coalgebras with a finitely generated carrier for an
endofunctor on a locally finitely presentable category. Their behaviour gives rise to a new fixpoint of the coalgebraic type
functor called \emph{locally finite fixpoint} (LFF). We prove that if the given
endofunctor preserves monos the LFF always exists and is a subcoalgebra of the
final coalgebra (unlike the rational fixpoint previously studied by Ad\'amek,
Milius and Velebil). Moreover, we show that the LFF is characterized by two
universal properties: 1.~as the final locally finitely generated coalgebra, and
2.~as the initial fg-iterative algebra. As examples of LFF's we first obtain the
known instances of the rational fixpoint, e.g. regular languages, rational
streams and formal power-series, regular trees etc. And we obtain a number of
new examples, e.g.~context-free languages, context-free formal power-series (and
any other instance of the generalized powerset construction by
\citeauthor*{generalizeddeterminization}) and the monad of Courcelle's algebraic
trees.
} % 2nd takeout
This paper contributes to a theory of the behaviour of ``finite-state''
systems that is generic in the system type. We propose that such systems are modeled as coalgebras with a finitely generated carrier for an
endofunctor on a locally finitely presentable category. Their behaviour gives rise to a new fixpoint of the coalgebraic type
functor called \emph{locally finite fixpoint} (LFF). We prove that if the given
endofunctor preserves monomorphisms then the LFF always exists and is a subcoalgebra of the
final coalgebra (unlike the rational fixpoint previously studied by Ad\'amek,
Milius and Velebil). Moreover, we show that the LFF is characterized by two
universal properties: 1.~as the final locally finitely generated coalgebra, and
2.~as the initial fg-iterative algebra. As instances of the LFF we first obtain the
known instances of the rational fixpoint, e.g. regular languages, rational
streams and formal power-series, regular trees etc. And we obtain a number of
new examples, e.g.~(realtime deterministic resp.~non-deterministic) context-free
languages, constructively $S$-algebraic formal power-series (and any other
instance of the generalized powerset construction by
\citeauthor*{generalizeddeterminization}) and the monad of Courcelle's algebraic
trees.
\end{abstract}
\section{Introduction}

Coalgebras capture many types of state based system within a uniform and
mathematically rich framework \cite{Rutten:2000:UCT:abbrev}. One outstanding feature of
the general theory is \emph{final semantics} which gives a fully abstract
account of system behaviour. For example, coalgebraic modelling of deterministic
automata (without a finiteness restriction on state sets) yields the set of all
formal languages as a final model, and restricting to \emph{finite} automata one
precisely obtains the regular languages \cite{Rutten:1998:ACE}. This
correspondence has been generalized to locally finitely presentable
categories~\cite{adamek1994locally,gu71}, where \emph{finitely presentable}
objects play the role of finite sets, leading to the notion of \emph{rational
fixpoint} that provides final semantics to all models with finitely presentable
carrier~\cite{streamcircuits}. It is known that the rational fixpoint is fully
abstract (identifies all behaviourally equivalent states) as long as finitely
presentable objects agree with finitely generated objects in the base
category~\cite[Proposition~3.12]{bms13}. While this is the case in some
categories (e.g. sets, posets, graphs, vector spaces, commutative monoids), it
is currently unknown in other base categories that are used in the
construction of system models, for example in idempotent semirings (used in the treatment of
context-free grammars \cite{coalgcontextfree}), in algebras for the stack monad (used for modelling configurations of stack machines~\cite{coalgchomsky}); or it even fails, for example in the category of finitary monads on sets (used in the categorical study
of algebraic trees \cite{secondordermonad}), or Eilenberg-Moore categories for a
monad in general (the target category of generalized determinization
\cite{generalizeddeterminization}, in which the above examples live).
Coalgebras over a category of Eilenberg-Moore algebras
over \Set in particular provide a paradigmatic setting: automata that describe
languages beyond the regular languages consist of a finite state set, but their
transitions produce side effects such as the manipulation of a stack. These can
be described by a monad, so that the (infinite) set of system states (machine
states plus stack content) is described by a free algebra (for that monad) that is
generated by the finite set of machine states. This is formalized by the
generalized powerset construction \cite{generalizeddeterminization} and
interacts nicely with the coalgebraic framework we present.

Technically, the shortcoming of the rational fixpoint is due to the
fact that finitely presentable objects are not closed under
quotients, so that the rational fixpoint itself may fail to be a
subcoalgebra of the final coalgebra and so identifies too little
behaviour. The main conceptual contribution of this paper is the
insight that also in cases where finitely presentable and finitely
generated do not agree, the \emph{locally finite fixpoint} provides
a fully abstract model of finitely generated behaviour. We give a construction
of the locally finite fixpoint, and support our claim both by
general results and concrete examples: we show that under mild
assumptions, the locally finite fixpoint always exists, and is
indeed a subcoalgebra of the final coalgebra. Moreover, we give a
characterization of the locally finite fixpoint as the initial iterative
algebra. We then instantiate our results to several scenarios
studied in the literature.

First, we show that the locally finite fixpoint is universal (and fully
abstract) for the class of systems produced by the generalized powerset
construction over \Set: every determinized finite-state system
induces a
unique homomorphism to the locally finite fixpoint, and the latter contains precisely
the finite-state behaviours.

Applied to the coalgebraic treatment of
context-free languages, we show that the locally finite fixpoint
yields precisely the context-free languages, and real-time
deterministic context-free languages, respectively, when modelled
using algebras for the stack monad of \cite{coalgchomsky}.
For context-free languages weighted in a semiring $S$, or
equivalently for constructively $S$-algebraic 
power series \cite{Petre:2009:ASP}, the locally finite fixpoint
comprises precisely those, by phrasing the results of \citet{jcssContextFree}
in terms of the generalized powerset construction.
Our last example shows the applicability of our results beyond
categories of Eilenberg-Moore algebras over \Set, and we characterize the
monad of Courcelle's algebraic trees over a signature \cite{courcelle,secondordermonad} as the
locally finite fixpoint of an associated functor (on a category of monads),
solving an open problem of~\cite{secondordermonad}.

The work presented here is based on the third author's master thesis in
\cite{wissmann2015}.
\iffull
Most proofs are omitted; they can be found in the appendix.
\else
Most proofs are omitted; they can be found in the full version~\cite{mpw15} of our paper.
\fi

\begin{notes}
\noindent\emph{Related Work.} The characterization of languages in
terms of (co-)algebraic constructions has been carried out for
various examples, such as (weighted) context-free languages
\cite{Winter:2013:CCC,coalgchomsky} as well as regular languages
\cite{Rutten:1998:ACE} where characterization theorems were
established on a case-by-case basis. We show that the locally finite
fixpoint provides a more general, and conceptual account. 
We have already mentioned the
rational fixpoint~\cite{iterativealgebras,streamcircuits} that
serves a similar purpose and shares many technical similarities with
the locally finite fixpoint, introduced here. Many of the properties
of the rational fixpoint in fact hold, \emph{mutatis
mutandis}, also for locally finite fixpoint, cf.~\emph{op.cit.}
\end{notes}

\takeout{ % old introduction
When comparing different notions of computation in theoretical computer
science, the core aspect is that of the expressive power of a computation model:
\textqt{Which behaviours can I program?}. Of course, the incarnations of those
programs vary from model to model:
\begin{itemize}
\item For \emph{finite state automatons}, programs are state sets together with a state transition table.
\item For \emph{push-down automata}, programs are state sets together with a state transition table that is allowed to read data from or push data to a stack
\item For \emph{context-free grammars}, programs are non-terminal symbols together with production rules.
\item \emph{Recursive program schemes} over a signature of \textqt{givens} are systems
of mutally recursive equations involving terms of \textqt{givens}.
\end{itemize}
When it comes to expressiveness considerations, both all these theoretical
models and programming languages in practise share an important restriction:
the program must be of finite size.

A generic way of considering those kind of models is by coalgebras,
i.e.~morphisms $c: C\to FC$ in a category $\C$ and for an endofunctor $F: \C\to
\C$. The choice of $\C$ and $F$ determine the computation model and a concrete
coalgebra $C,c$ represents a concrete automaton, grammar, or program
repsectively. All the above examples have coalgebraic characterizations, and
are recalled later.

The collection of \emph{all} the possible behaviours $F$-coalgebras (regardless
on the size of $C$ or behaviour of $c$) is represented by the \emph{final}
$F$-coalgebra (if it exists). This means that the final coalgebra also contains
behaviours that can only be described by infinite programs -- recall for
example that any formal language is accepted by a deterministic automaton with
infinitely many states, namely its word automaton. Still, the final coalgebra
offers us a notion of equality between behaviours: two behaviours are
equivalent iff they are identified in the final coalgebra. E.g.~two formal
languages are the same iff they have the same word automaton.

In order to talk about finite automatons, \citet{iterativealgebras} developed
the rational fixpoint of a functor $F$: a coalgebra that contains the
behaviours of all coalgebras with a finitely presentable $C$. Here, finitely
presentable is a categorical generalization of finiteness. However, the
rational fixpoint lacked the second property of the final coalgebra: it does
not identify identical behaviours.

The background problem is that finitely presentable objects are not closed
under quotients in general.
} % end takeout
\takeout{ % newer introduction
When modelling computational phenomena by state-based systems or by a recursive
system of equations one usually aims for a finite representation, e.g.~in
automata theory one considers (non-deterministic, weighted, probabilistic etc.)
automata with a finite state set, context-free grammars are essentially a system
of recursive equations with finitely many variables (the non-terminals) and
linear systems are supposed to have a finite-dimensional state space. Another
example are regular and algebraic trees for a signature which arise as the
solutions of systems of finitely many recursive term equations and 1st-order
recursive program schemes, respectively (see~\cite{courcelle}). 

Coalgebras are known to capture many types of state-based systems and recursive
specifications within a uniform and mathematically pleasing general theory. To
model a certain type of systems one fixes an endofunctor $F$ on a category $\C$
describing the transition type of the class of systems of interest. One system
is then modeled as an $F$-coalgebra, i.e.~a morphism $S \to FS$ in $\C$ where
$S$ is the object of states of the system. Coalgebras for $F$ always come
(under mild assumptions on the type functor $F$) with a canonical domain of
behaviour, the \emph{final} $F$-coalgebra $\nu F$. Its universal property allows
to assign to every state of a coalgebra $S \to FS$ its ``behaviour'' via the
unique $F$-coalgebra homomorphism $S \to \nu F$, and this also gives rise to the
canonical notion of behavioural equivalence of two states. For example,
coalgebras for the functor $FX = \{0,1\} \times X^\Sigma$ where $\Sigma$ is a
finite input alphabet are precisely deterministic automata on $\Sigma$, and the
final coalgebra is carried by the set of all formal languages. The unique
homomorphism from an automaton (considered as an $F$-coalgebra) into the final
coalgebra assigns to every state of the automata the language it accepts. 

%\smnote{todo example, det.~autom.} \dots
But usually, one is interested in the behaviour of ``finite-state'' systems of a
certain type (for a example \emph{regular} languages over the alphabet
$\Sigma$). This behaviour is captured in the theory of coalgebras by the
\emph{rational fixpoint} of $\varrho F$ proposed by Ad\'amek, Milius and Velebil
(see~\cite{iterativealgebras,streamcircuits}). The idea is to take as the notion
of ``finite object'' (of states) the notion of a finitely presentable object in
a locally finitely presentable category (these categories were introduced by
Gabriel and Ulmer~\cite{gu71}, see also~\cite{adamek1994locally}). The theory of
the rational fixpoint is by now well-established; its parametrized version
originally arose as a coalgebraic aproach to the study of the semantics of
finitary recursive equations and Elgot's iterative theories
(see~\cite{iterativealgebras}), it lies at the heart of the coalgebraic study of
Bloom's and \'Esik's iteration theories~\cite{amv_em1,amv_em2}, and it plays a
central role in coalgebraic regular expression calculi and axiomatization of
expression equivalence~\cite{streamcircuits,bms13}. 

However, while in several concrete applications the rational fixpoint captures
precisely the intended regular behaviour this may fail in general. More
precisely, it is well-known that when the classes of finitely presentable (fp)
and finitely generated (fg) objects in the lfp category $\C$ coincide, then
$\varrho F$ is a subcoalgebra of $\nu F$ (see~\cite[Proposition~3.12]{bms13}). But
there are cases where $\varrho F$ does \emph{not} identify behaviourally equivalent
states (see~\cite[Proposition~3.15]{bms13}), and so $\varrho F$ fails to be the
desired domain of regular behaviour. Note also that while for some base
categories $\C$ (such as sets, posets, graphs, vector spaces, presheaves on
finite sets, commutative monoids, and semimodules for Noetherian semirings) it
is known that fp~and fg~objects coincide, in some categories this is known not
to hold (e.g.~groups, monoids, finitary set monads) and most often this is
unknown (e.g.~for many varieties of algebras such as idempotent semirings).
However, some of these categories feature as base categories in recent
coalgebraic work, e.g.~idempotent semirings for context-free grammars (see
Winter et al.~\cite{coalgcontextfree}) and finitary set monads for the study of
algebraic trees for a signature (see~\cite{secondordermonad}). 

Our contribution in this paper is a new approach to regular behaviour of
$F$-coalgebras based on taking fg~objects (in lieu of fp~ones) as ``finite
object'' of states. After recalling some preliminaries in
Section~\ref{sec:prelim}, we study in Section~\ref{sec:lff} the behaviour of all
coalgebras with an fg carrier. For a given finitary and monomorphism preserving
endofunctor $F$ on an lfp category $\C$ we define a coalgebra $\LFF F$ as the
filtered colimit of all fg carried $F$-coalgebras. We prove that $\LFF F$ is
always a subcoalgebra of the final coalgebra and a fixpoint of $F$ and call it
the locally finite fixpoint. Next two universal properties of $\LFF F$ are
proved: (1) $\LFF F$ is the final locally finitely generated (lfg) coalgebra
and, inverting its structure morphism, (2) the initial fg-iterative algebra for $F$
-- these universal properties are similar to the two universal properties of
the rational fixpoint. Under additional assumptions we prove that $\LFF F$ is
the image of $\varrho F$ in the final coalgebra, and whenever fp $=$ fg objects
then $\LFF F \cong \varrho F$. Section~\ref{sec:app} presents several
applications of the LFF: besides all previous applications of the rational
fixpoint we present context-free languages, context-free formal power series,
the behaviour of (non-deterministic) stack $T$-automata of Goncharov et
al.~\cite{coalgchomsky} and Courcelle's algebraic trees. The last application
solves an open problem from~\cite{secondordermonad} to characterize the monad of
algebraic trees by a universal property. Finally, Section~\ref{sec:con}
concludes the paper.

\takeout{ % Stoffsammlung
\begin{itemize}
\item Why Interesting? (machines in TCS always finite-state-machines +
side-effects, finitely many recursive equations, ...)
\item Existing Work? The rational fixpoints? Not a subcoalgebra, Image described
by universal property?
\item Question whether fp=fg is hard/tedious to answer, so we need a uniform
framework that works in both cases (and even if we don't know).
\end{itemize}
}%
} % takeout
\section{Preliminaries and Notation}
\label{sec:prelim}

\partitle{Locally finitely presentable categories.} 
A \emph{filtered colimit} is the colimit of a diagram
$\D \to \C$ where $\D$ is filtered (every finite
subdiagram has a cocone in $\D$) and \emph{directed} if $\D$ is additionally a
poset. \emph{Finitary functors} preserve filtered (equivalently
directed) colimits. Objects $C \in \C$ are \emph{finitely
presentable} (fp) if the hom-functor $\C(C, -)$ preserves filtered
(equivalently directed) colimits, and \emph{finitely generated}
(fg) if $\C(C, -)$ preserves directed colimits of monos (i.e.~colimits of
directed diagrams where all connecting morphisms are
monic). Clearly any fp object is fg, but not vice versa. Also, fg
objects are closed under strong epis (quotients) which fails for fp
objects in general.
A cocomplete category is \emph{locally finitely presentable} (lfp) if
the full subcategory $\C_\fp$ of finitely presentable objects is
essentially small, i.e.~is up to isomorphism only a set, and every object $C \in
\C$ is a filtered colimit of a diagram in $\C_\fp$. We refer
to~\cite{gu71,adamek1994locally} for further details.

It is well known that the categories of sets, posets and graphs are
lfp with finitely presentable objects precisely the finite
sets, posets, graphs, respectively.
The category of
vector spaces is lfp with finite-dimensional spaces being fp. Every
finitary variety is lfp (i.e.~an equational class of algebras induced by
finite-arity operations or equivalently the Eilenberg-Moore category for a
finitary \Set-Monad, see Section~4.1 %\ref{sec:powerset} 
later). The finitely
generated objects are the finitely generated algebras, and finitely presentable
objects are algebras specified by finitely many generators and relations. This
includes the categories of groups, monoids, (idempotent) semirings,
semi-modules, etc. Every lfp category has mono/strong epi
factorization~\cite[Proposition 1.16]{adamek1994locally}, i.e.~every $f$ factors
as $f = m \o e$ with $m$ mono (denoted by $\rightarrowtail$), $e$ strong epi
(denoted by $\twoheadrightarrow$), and we call the domain $\Im(f)$ of $e$
the \emph{image} of $f$. Any strong epi $e$ has the diagonal fill-in
property, i.e. $m \o g = h \o e$ with $m$ mono and $e$ strong epi gives a unique
$d$ such that $m \o d = h$ and $g = d \o e$.

\takeout{
  \twnotei{What's the point of this item? In my opinion, we are not
  interested in a sufficient criterion for fp=fg. Furthermore, we look at
  modules, but not necessarily of notherian rings. So I vote for dropping
  without substitution.} Modules for a Noetherian semiring. Recall that a (semi-)module for a
  semiring $\Sd$ is a commutative monoid $(M,+,0)$ together with an action of
  the semiring $\Sd$ on $M$ satisfying the usual distributive laws $r(m+n) =
  rm+rn$ and $r0 = 0$. Hence, modules for $\Sd$ form a finitary variety. In
  general the classes of fg modules and fp modules do not coincide. The semiring
  $\Sd$ is called \emph{Noetherian} if any submodule of a finitely generated
  module is itself finitely generated. For Noetherian semirings the classes of
  finitely generated and finitely presentable modules coincide (see
  e.g.~\cite[Prop.~2.6]{bms13} for a proof). There are also non-Noetherian
  semirings for which fp and fg modules coincide; e.g.~the module for the
  semiring of natural numbers for which modules are precisely the commutative
  monoids. 
}

\noindent
\partitle{Coalgebras.} If $H: \C \to \C$ is an endofunctor,
\emph{$H$-coalgebras} are pairs $(C, c)$ with $c: C \to HC$, and $C$
is the \emph{carrier} of $(C, c)$. Homomorphisms $f: (C, c) \to (D, d)$ are maps
$f: C \to D$ such that $Hf \o c = d \o f$. This gives a category denoted by
$\Coalg H$. If its final object exists then this final $H$-coalgebra $(\nu
H,\tau)$ represents a canonical domain of behaviours of $H$-typed systems, and
induces for each $(C,c)$ a unique homomorphism, denoted by $c^\dagger$, giving
semantics to the system $(C,c)$. The final coalgebra always exists provided $\C$
is lfp and $H$ is finitary. The forgetful functor $\Coalg H \to \C$ creates
colimits and reflects monos and epis. A morphism $f$ in $\Coalg H$ is
\emph{mono-carried} (resp.~\emph{epi-carried}) if the underlying morphism in
$\C$ is monic (resp.~epic). Strong epi/mono factorizations lift from $\C$ to
$\Coalg H$ whenever $H$ preserves monos yielding epi-carried/mono-carried
factorizations. A \emph{directed union of coalgebras} is the colimit of a directed diagram in
$\Coalg H$ where all connecting morphisms are mono-carried.

\noindent\partitle{The Rational Fixpoint.} For $\C$ lfp and $H: \C \to
\C$ finitary let $\Coalg_\fp H$ denote the full subcategory of
$\Coalg H$ of coalgebras with fp carrier, and $\Coalg_\lfp H$ the
full subcategory of $\Coalg H$ of coalgebras that arise as filtered
colimits of coalgebras with fp carrier~\cite[Corollary
III.13]{streamcircuits}. The coalgebras in $\Coalg_\lfp H$ are called \emph{lfp coalgebras} and for $\C = \Set$ those are precisely the locally finite coalgebras (i.e.~those coalgebras where every element is contained in a finite subcoalgebra). The final lfp coalgebra exists and is the
colimit of the inclusion $\Coalg_\fp H \hookrightarrow \Coalg H$,
and it is a fixpoint of $H$ (see~\cite{iterativealgebras}) called the \emph{rational fixpoint} of $H$. Here are some examples: the rational fixpoint of a polynomial set functor associated to a finitary signature $\Sigma$ is the set of rational
$\Sigma$-trees~\cite{iterativealgebras}, i.e.~finite and infinite
$\Sigma$-trees having, up to isomorphism, finitely many subtrees
only, and one obtains rational weighted languages for Noetherian semirings $S$ for a functor on the category of $S$-modules~\cite{bms13}, and rational $\lambda$-trees
for a functor on the category of presheaves on finite sets~\cite{highrecursion} or for a related functor on nominal sets~\cite{MiliusWissmannRatlambda}. 
If the classes of fp and fg objects coincide in $\C$, then the rational fixpoint
is a subcoalgebra of the final coalgebra~\cite[Theorem~3.12]{bms13}. This is the
case in the above examples, but not in general, see~\cite[Example~3.15]{bms13}
for a concrete example where the rational fixpoint does not identify
behaviourally equivalent states. Conversely, even if the classes differ, the
rational fixpoint can be a subcoalgebra, e.g.~for any constant functor.

\noindent\partitle{Iterative Algebras.} If $H: \C \to \C$ is an
endofunctor, an $H$-algebra $(A, a: HA \to A)$ is \emph{iterative}
if  every \emph{flat equation morphism} $e: X \to HX + A$ where $X$
is an fp object has a unique \emph{solution}, i.e.~if there
exists a unique $e^\dagger: X \to A$ such that $e^\dagger =
[a,\id_A] \o (He^\dagger + \id_A) \o e$. The rational
fixpoint is also characterized as the initial iterative algebra~\cite{iterativealgebras} and is the
starting point of the coalgebraic approach to Elgot's iterative
theories~\cite{elgot} and to the iteration theories of Bloom and
\'Esik~\cite{be,iterativealgebras,amv_em1,amv_em2}.

\takeout{
\paragraph{The Rational Fixpoint.} Before we present our main results on the
locally finite fixpoint in the next section let us now first recall some facts
about the rational fixpoint from~\cite{iterativealgebras,streamcircuits}. Assume
that $F: \C\to\C$ is a finitary endofunctor on an lfp category. Then one
considers all $F$-coalgebras with a finitely presentable carrier; they are
supposed to capture all systems with a ``finite'' objects of states, and they
form the full subcategory $\Coalgfp F\hookrightarrow \Coalg F$. We further
consider \emph{locally finitely presentable} (\emph{lfp}, for
short) $F$-coalgebras. On $\Set$, they capture precisely
locally finite coalgebras$\mathrlap{\text{,}}$\footnote{leading to the general notion \emph{lfp}
coalgebra, not to be confused with lfp category.} i.e.~those coalgebras where every element of the
carrier lies in a finite subcoalgebra. We do not recall the general definition
but recall from~\cite[Corollary III.13]{streamcircuits} that lfp coalgebras are
characterized as precisely those $F$-coalgebras arising as filtered colimits of
coalgebras from $\Coalgfp F$. It then follows that a final
lfp coalgebra $r: \varrho F \to F(\varrho F)$ exists and that it is constructed
as the colimit of the filtered diagram $\Coalgfp F \hookrightarrow \Coalg F$ of
\emph{all} fp-carried $F$-coalgebras. Furthermore, one can prove that $\varrho
F$ is a fixpoint of $F$ (see~\cite[Lemma~3.4]{iterativealgebras}). The ensuing
$F$-algebra $r^{-1}: F(\varrho F) \to \varrho F$ also has a universal property,
too. This algebra is the \emph{initial iterative} $F$-algebra, where an
$F$-algebra $a: FA \to A$ is called \emph{iterative} if every \emph{flat
equation morphism} $e: X \to FX + A$ where $X$ is an fp object has a unique
\emph{solution}, i.e.~given $e$ there exists a unique morphism $e^\dagger: X \to
A$ such that $e^\dagger = [a,A] \o (Fe^\dagger + A) \o e$. This universal
property of $\varrho F$ leads to more powerful finitary corecursion schemes,
solutions theorems and specification principles; this has been explored
e.g.~in~\cite{iterativealgebras,bmr12,mbmr13}. It has also been the starting
point of the coalgebraic approach to Elgot's iterative theories~\cite{elgot} and
also the iteration theories of Bloom and \'Esik~\cite{be}
(see~\cite{iterativealgebras,amv_em1,amv_em2}).

\begin{example}
Prominent examples of the rational fixpoint are:
\begin{enumerate}
\item For a signature functor $H_\Sigma$ on \Set, we get the rational
$\Sigma$-trees as $\varrho H_\Sigma$, i.e.~those $\Sigma$-trees with only
finitely many subtrees (up to isomorphism).
\item Rational weighted languages for Noetherian semirings.
\item Rational $\lambda$-trees. two times: in Nom and $\Set^{\mathcal F}$
\end{enumerate}
\end{example}

In all the above examples the rational fixpoint appears as a subcoalgebra of the
final coalgebra, that means that it collects precisely all behaviours of the
coalgebras from $\Coalgfp F$ modulo behavioural equivalence. It has been proved
that this happens whenever the classes of fp and fg objects coincide in $\C$
(see~\cite[Theorem~3.12]{bms3}). However, in arbitrary lfp categories (or
finitary varieties) this is sometimes not true (e.g.~in the categories of
groups, monoids, Heyting algebras or finitary monads on $\Set$), and most often
this seems to be unknown, e.g.~for idempotent semirings. Moreover, in those
cases where it is known the proofs tend to be rather non-trivial making use of
deep results in algebra or somewhat involved combinatorial arguments.\smnote{add
citations here?} And if the two classes of objects do not coincide it may happen
that the rational fixpoint is \emph{not} a subcoalgebra of the final coalgebra,
i.e.~there exist behaviourally equivalent states that are not identified in
$\varrho F$ (see~\cite[Example~3.15]{bms13} for a concrete example). 
} % takeout

\section{The Locally Finite Fixpoint}

The locally finite fixpoint can be characterized similarly to the
rational fixpoint, but with respect to coalgebras with finitely
generated (not finitely presentable) carrier. We show that the
locally finite fixpoint always exists, and is a subcoalgebra of the
final coalgebra, i.e. identifies all behaviourally equivalent
states. As a consequence, the locally finite fixpoint provides a
fully abstract notion of finitely generated behaviour. From now on, we rely on
the following:

\label{sec:lff}
\begin{assumption} \label{basicassumption}
    Throughout the rest of the paper we assume that $\C$ is an lfp category and
    that $H:\C→\C$ is finitary and preserves monomorphisms.
\end{assumption}

\noindent
As for the rational fixpoint, we denote the full subcategory of
$\Coalg H$ comprising all coalgebras with finitely generated
carrier by $\Coalgfg H$ and have the following notion of locally finitely
generated coalgebra.

%\subsection{Locally Finitely Generated Coalgebras}

\begin{definition}
    \label{lfgcoalgebra}
    A coalgebra $X\xrightarrow{x} HX$ is called \emph{locally finitely
    generated (lfg)} if for all $f: S\rightarrow
    X$ with $S$ finitely generated, there exist a coalgebra $p: P
    \to HP$ in $\Coalgfg H$, a coalgebra morphism $h: (P,p) \rightarrow (X,x)$ and some $f': S\rightarrow P$ such that $h \o f' = f$. 
    %\[
    %    \begin{tikzcd}
    %        S \arrow[->]{rr}{f} \arrow[->]{dr}[below left]{f'} & {} & X \\
    %        {} & P\arrow[->]{ur}[below right]{h}
    %        \descto{u}{\circlearrowleft}
    %    \end{tikzcd}
    %\]
    $\Coalglfg H\subseteq \Coalg H$ denotes the full subcategory of lfg coalgebras.
\end{definition}
%
%\noindent
Equivalently, one can characterize lfg coalgebras in terms of subobjects and subcoalgebras,
making it a generalization of of \emph{local finiteness} in \Set, i.e.~the property of a coalgebra that every element is contained in a finite subcoalgebra. 
\begin{lemma}
    $X\xrightarrow{x} HX$ is an lfg coalgebra iff for all 
    fg subobjects $S\smash{\,\xrightarrowtail{f}\,} X$, there exist
    a subcoalgebra $h: (P,p) \rightarrowtail (X,x)$ and a mono $f':
    S\rightarrowtail P$ with~$h\cdot f' = f$,
    i.e.~$S$ is a subobject of $P$. 
\end{lemma}
\begin{proof}
    ($\Rightarrow$) Given some mono $f: S\rightarrowtail X$,
    factor the induced $h$ into some strong epi-carried and mono-carried homomorphisms
    and use that fg objects are closed under strong epis. 
    ($\Leftarrow$) Factor $f: S\to X$ into an epi and a mono $g: \Im(f)
    \rightarrowtail X$ and use the diagonal fill-in property for $g$.
    \qed
\end{proof}

\noindent
Evidently all coalgebras with finitely generated carriers are lfg.
Moreover, lfg coalgebras are precisely the filtered colimits of coalgebras from $\Coalgfg H$. 

\begin{proposition}\label{prop:lfgcolim}
    Every filtered colimit of coalgebras from $\Coalgfg H$ is lfg.
\end{proposition}
\begin{proof}[\iffull Sketch; for the full proof see the appendix\else Sketch\fi]
  One first proves that directed unions of coalgebras from $\Coalgfg H$ are lfg.
  Now given a filtered colimit $c_i: X_i \to C$ where $X_i$ are coalgebras in
  $\Coalgfg H$, one epi-mono factorizes every colimit injection: $c_i =
  (\!\!\smash{{\begin{tikzcd}[column sep=5mm]
  X_i \arrow[->>,yshift=-1pt]{r}{e_i\,} \pgfmatrixnextcell T_i \arrow[>->,yshift=-1pt]{r}{m_i}
  \pgfmatrixnextcell C
  \end{tikzcd}}}\!\!)$. Using the diagonalization of the
  factorization one sees that the $T_i$ form a directed diagram of subobjects of
  $C$. Furthermore $C$ is the directed union of the $T_i$ and therefore an lfg
  coalgebra as desired.  
%    Factor the colimit injections, obtaining subcoalgebras with a fg carrier and
%    connecting mono-carried coalgebra homomorphisms. In total: a directed union,
%    that is by construction cofinal. By the definition of lfg coalgebra and fg
%    object, the colimit of a directed union is lfg.
    \qed
\end{proof}
\begin{proposition} \label{lfgdirectedunion}
    Every lfg coalgebra $(X,x)$ is a directed colimit of its subcoalgebras from
    $\Coalgfg H$.
\end{proposition}
\begin{proof}
    Recall from \cite[Proof I of Theorem 1.70]{adamek1994locally} that $X$ is
    the colimit of the diagram of all its finitely generated subobjects. Now the subdiagram given by all subcoalgebras of $X$ is cofinal. Indeed, this follows directly from the fact that $(X,x)$ is an lfg coalgebra: for every
    subobject $S\rightarrowtail X$, $S$ fg, we have a subcoalgebra of $(X,x)$
    in $\Coalgfg H$ containing $S$.
    \qed
\end{proof}

\begin{corollary}\label{lfgcharacterization}
The lfg coalgebras are precisely the filtered colimits, or
equivalently directed unions, of coalgebras with fg~carrier.
\end{corollary}

\noindent
As a consequence, a coalgebra is final in $\Coalglfg F$ if there is
a unique morphism from every coalgebra with finitely generated carrier.

\begin{proposition}\label{finalforfg}
    An lfg coalgebra $L$ is final in $\Coalglfg H$ iff 
    for every for every coalgebra $X$ in $\Coalgfg H$ there exists a unique coalgebra morphism from $X$ to $L$. 
\end{proposition}
The proof is analogous to \cite[Theorem 3.14]{streamcircuits}; the
full argument can be found in\iffull\ the Appendix.\else~\cite{mpw15}.\fi\ Cocompleteness of $\C$
ensures that the final lfg coalgebra always exists.
\begin{theorem}\label{thm:final}
    The category $\Coalglfg H$ has a final object, and the final
    lfg coalgebra is the colimit of the inclusion $\Coalgfg H \hookrightarrow \Coalglfg H$.
\end{theorem}
\begin{proof}
    By \autoref{lfgcharacterization}, the colimit of the inclusion $\Coalgfg H
    \hookrightarrow \Coalglfg H$ is the same as the colimit of the entire
    $\Coalglfg H$. And the latter is clearly the final object of $\Coalglfg H$.\qed
\end{proof}
This theorem provides a construction of the final lfg coalgebra collecting
precisely the behaviours of the coalgebras with fg carriers. In the following we
shall show that this construction does indeed identify precisely behaviourally equivalent states, i.e.~the final lfg coalgebra is always a subcoalgebra of the final coalgebra. 
Just like fg objects are closed under quotients -- in contrast to fp objects -- we have
a similar property of lfg coalgebras:
\begin{lemma}\label{lfgquotients}
    Lfg coalgebras are closed under strong quotients, i.e.~for every strong epi carried
    coalgebra homomorphisms $X \twoheadrightarrow Y$, if $X$ is lfg then so is $Y$.
\end{lemma}
The failure of this property for lfp coalgebras is the reason why the
rational fixpoint is not necessarily a subcoalgebra of the final coalgebra and
in particular the rational fixpoint in \cite[Example 3.15]{bms13} is an lfp
coalgebra for which the property fails.
\begin{theorem}
    The final lfg $H$-coalgebra is a subcoalgebra of the final $H$-coalgebra.
\end{theorem}
\begin{proof}
  Let $(L,\ell)$ be the final lfg coalgebra. Consider the unique coalgebra morphism $L \to \nu H$ and take its factorization:  
    \[
        \begin{tikzcd}
            (L,\ell) \arrow[->>,yshift=1mm]{r}{e}
                \arrow[loop left,>->,dashed,looseness=8]{}{\id}
            &(I,i) \arrow[>->]{r}{m}
                \arrow[dashed,yshift=-1mm]{l}{i^\dagger}
            &(\nu H,\tau)
        \end{tikzcd},
        \quad\text{with $e$ strong epi in $\C$}.
    \]
    By \autoref{lfgquotients}, $I$ is an lfg coalgebra and so by finality of $L$ we have the coalgebra morphism $i^\dagger$ such that $\id_L = i^\dagger\cdot e$. It follows that $e$ is monic and therefore an iso.
    \qed
\end{proof}

\noindent
In other words, the final lfg $H$-coalgebra collects precisely the finitely generated
behaviours from the final $H$-coalgebra. We now show that the final
lfg coalgebra is a fixpoint of $H$ which hinges on the following:
\begin{lemma}\label{Hlfg}
    For any lfg coalgebra $C\xrightarrow{c} HC$, the coalgebra
    $HC\xrightarrow{Hc}{HHC}$ is lfg.
\end{lemma}
\begin{proof}
    Consider $f: S\to HC$ with $S$ finitely generated. As \C is lfp we know that
    $HC$ is the colimit of its fg subobjects, and so $f: S\to
    HC$ factors through some subobject $\inj_q: Q\rightarrowtail HC$ with $Q$ fg and $f =
    \inj_q\cdot f'$. On the other hand, $(C,c)$ is lfg, i.e.~the directed union
    of its subcoalgebras from $\Coalgfg H$. Then, since $H$ is finitary and
    mono-preserving,
    $HC \xrightarrow{c} HHC$ is also a directed union and the morphism $\inj_q: Q\to HC$
    factors through some $HP\xrightarrow{Hp}HHP$ with $(P,p)\in \Coalgfg H$ via $\inj_p: (P,p)\rightarrowtail (C,c)$,
    i.e.~$H\inj_p\cdot q=\inj_q$. Finally, we can
    construct a coalgebra with fg carrier
    \[
        Q + P
        \xrightarrow{[q,p]}
        HP
        \xrightarrow{H\inr}
        H(Q+P)
    \]
    and a coalgebra homomorphism $H\inj_p\cdot[q,p]: Q+P \to HC$. In the diagram
    \[
    \begin{tikzcd}
    S
    \arrow{rr}{f}
    \arrow{dd}[left]{f'}
    &&
        HC \arrow{rr}{Hc}
        &&
        HHC
        \\
        &&
        HP \arrow{rr}{Hp}
        \arrow[oldequal]{dr}
        \arrow{u}[right]{H\inj_p}
        &&
        HHP
        \arrow{u}[right]{HH\inj_p}
        \\
        Q
        \arrow{uurr}[sloped,above]{\inj_q}
        \arrow{urr}[sloped,above]{q}
        \arrow{rr}{\inl}
        &&
        Q+P
        \arrow{r}[pos=0.4]{[q,p]}
        \arrow{u}[left]{[q,p]}
        &
        HP
        \arrow{r}{H\inr}
        &
        H(Q+P)
        \arrow{u}[right]{H[q,p]}
        \arrow[shiftarr={xshift=1.5cm}]{uu}[right]{H(H\inj_p\cdot[q,p])}
    \end{tikzcd}
    \]
    every part trivially commutes, so $H\inj_p\cdot[q,p]$ is the
    desired homomorphism.
    \qed
\end{proof}
So with a proof in virtue to Lambek's Lemma \cite[Lemma 2.2]{lambek}, we obtain the desired fixpoint: 
\begin{theorem}
    The carrier of the final lfg $H$-coalgebra is a fixpoint of $H$. 
\end{theorem}
We denote the above fixpoint by $(\LFF H, \ell)$ and call it the \emph{locally finite fixpoint} (LFF) of $H$.
In particular, the LFF always exists under \autoref{basicassumption}, providing a finitary corecursion principle.

\iffull\else
\begin{remark}
  As we mentioned in the introduction the rational fixpoint of the finitary functor $H$ is the initial iterative algebra for $H$. A similar algebraic characterization is possible for the LFF. One simply replaces the fp object $X$ in the definition of a flat equation morphism by an fg object to obtain the notion of an \fgiterative algebra. 
\end{remark}
\begin{theorem}
  The LFF is the \emph{initial \fgiterative} $H$-algebra. 
\end{theorem}
For details, see the full version~\cite{mpw15} of our paper or~\cite{wissmann2015}. 
\fi%
\iffull
\subsection{Iterative Algebras}
Recall from~\cite{iterativealgebras,streamcircuits} that the rational fixpoint of a functor $H$ has a universal property both as a coalgebra and as an algebra for $H$. This situation is completely analogous for the LFF. We already established its universal property as a coalgebra in~\autoref{thm:final}. Now we turn to study the LFF as an algebra for $H$. 
%
%One can extend this to a more powerful corecursion principle, which represents the
%basics of Elgot's iterative theories \cite{iterativealgebras}. A corecursive definition is formalized
%by an equation morphism:
\begin{definition}
    An \emph{equation morphism} $e$ in an object $A$ is a morphism
    \(
        X \to HX + A,
    \)
    where $X$ is a finitely generated object. If $A$ is the carrier of an
    algebra $\alpha: HA\to A$, we call the \C-morphism $e^\dagger: X \to A$ a
    \emph{solution} of $e$ if $[\alpha,\id_A]\cdot He^\dagger + \id_A \cdot e = e^\dagger$.
    An $H$-algebra $A$ is called \emph{\fgiterative} if every equation morphism in $A$
    has a unique solution.
\end{definition} 

\begin{example}[{see \cite[Example 2.5 (iii)]{m_cia}}]
    The final $H$-coalgebra (considered as an algebra for $H$) is \fgiterative. In fact, in this algebra even morphisms $X \to HX + \nu H$ where $X$ is not necessarily an fg object have a unique solution. 
\end{example}

\begin{definition}
    For \fgiterative algebras $A$ and $C$, an equation morphism $e: X \to HX +
    A$ and a morphism $h: A\to C$ of $\C$ define an equation morphism $h\bullet e$ in $C$ as
    \(
        \begin{tikzcd}
            X \arrow{r}{e} &
            HX + A \arrow{r}{HX+h} & HX + C.
        \end{tikzcd}
    \)
    We say that $h$ preserves the solution $e^\dagger$ of $e$ if
    \(
        h\cdot e^\dagger = (h\bullet e)^\dagger.
    \)
    The morphism $h$ is called \emph{solution preserving} if it
    preserves the solution of any equation morphism $e$.
\end{definition}
Similarly to \cite{iterativealgebras}, the algebra homomorphisms are precisely
the solution preserving morphisms between iterative algebras, the proof is also
very similar.

\begin{proposition}
    \label{LFFisIterative}
    The locally finite fixpoint is \fgiterative.
\end{proposition}
%\begin{proof}[Sketch]
%    One can transform $e: X\to HX+\LFF H$ into an lfg coalgebra on $X + \LFF H$.
%    Then one shows that there is a one-to-one correspondence between
%    homomorphisms into $\LFF H$ and solutions of $e$ in the algebra $\ell^{-1}: H(\LFF H) \to \LFF H$ by
%    diagram chasing.
%    \qed
%\end{proof}

\begin{theorem}\label{alllfginitial}
    For an \fgiterative algebra $\alpha: HA \to A$ and an lfg coalgebra $e:
    X\to HX$ there is a unique $\C$-morphism $u_e: X\to A$ such that $u_e
    = \alpha \cdot Hu_e \cdot e$.
\end{theorem}
\begin{corollary}
    The locally finite fixpoint is the initial fg-iterative algebra.
\end{corollary}
\fi

\iffull
\subsection{Relation to the Rational Fixpoint}
\else
\medskip
\noindent\partitle{Relation to the Rational Fixpoint.} %
\fi
There are examples, where the rational fixpoint is not a subcoalgebra of the
final coalgebra. In categories, where fp~and fg~objects coincide, the
rational fixpoint and the LFF coincide as well (cf.~the respective colimit-construction in Section~\ref{sec:prelim} and Theorem~\ref{thm:final}). In this section we will see,
under slightly stronger assumptions, that fg-carried coalgebras are quotients of
fp-carried coalgebras, and in particular the locally finite fixpoint is a
quotient of the rational fixpoint: namely its image in the final coalgebra.

\begin{assumption}\label{projectiveassumption}
    In addition to \autoref{basicassumption}, assume that in the base category
    $\C$, every finitely presentable object is a strong quotient of a finitely
    presentable strong epi projective object and that the endofunctor $H$ also
    preserves strong epis.
\end{assumption}
The condition that every fg object is the strong quotient of a strong epi
projective often is phrased as \emph{having enough strong epi
projectives}~\cite{borceux1994handbook}. This assumption is apparently very strong but still is met 
in many situations:
\begin{example}
\begin{itemize}
    \item In categories in which all (strong) epis are split, every object is
    projective and any endofunctor preserves epis, e.g.~in \Set or $\Vec_K$.

    \item In the category of finitary endofunctors $\Funf(\Set)$, all
    polynomial functors are projective. The finitely presentable functors are
    quotients of polynomial functors $H_\Sigma$, where $\Sigma$ is a finite
    signature.

    \item In the Eilenberg-Moore category $\Set^T$ for a finitary monad $T$,
    strong epis are surjective $T$-algebra homomorphisms, and thus preserved by
    any endofunctor. In $\Set^T$, every free algebra $TX$ is projective; this is easy to see using the
    projectivity of $X$ in \Set. Every finitely generated object of $\Set^T$ is
    a strong quotient of some free algebra $TX$ with $X$ finite. For
    more precise definitions, see Section~4.1 %\ref{sec:powerset} 
    later.
\end{itemize}
\end{example}

\begin{proposition}\label{lfpquotient}
    Every coalgebra in $\Coalgfg H$ is a strong quotient of a coalgebra with
    finitely presentable carrier.
\end{proposition}

\begin{theorem}
    $\LFF H$ is the image of the rational fixpoint $\varrho H$ in the final coalgebra.
\end{theorem}
\vspace{-2mm}
\begin{proof}
    Consider the factorization $(\varrho H,r)
    \overset{e}{\twoheadrightarrow} (B,b) \overset{m}{\rightarrowtail} (\nu
    H,\tau)$. Since $\varrho H$ is the colimit of all fp carried $H$-coalgebras it is an lfg coalgebra by~\autoref{prop:lfgcolim} using that fp objects are also fg. Hence, by~\autoref{lfgquotients} the coalgebra $B$ is lfg, too. By
    \autoref{finalforfg} it now suffices to show that from every $(X,x) \in \Coalgfg H$ there exists a unique coalgebra morphism into $(B,b)$. Given $(X,x)$ in $\Coalgfg H$, it is the quotient  $q: (P,p)
    \twoheadrightarrow (X,x)$ of an fp-carried coalgebra by
    \autoref{lfpquotient}. Hence, we obtain a unique coalgebra morphism
    $p^\dagger: (P,p) \to (\varrho H,r)$. By finality of $\nu H$, we have
    $m\cdot e\cdot p^\dagger = x^\dagger \cdot q$ (with $x^\dagger: (X,x)\to (\nu
    H,\tau)$). So the diagonal fill-in property induces a homomorphism $(X,x) \to (B,b)$, being the only homomorphism $(X,x)
    \to (B,b)$ by the finality of $\nu H$ and because $m$ is monic. 
    \qed
\end{proof}
\vspace{-2mm}

\section{Instances of the Locally Finite Fixpoint}
\label{sec:app}
We will now present a number of instances of the LFF. First note, that all the
known instances of the rational fixpoint (see e.g.~\cite{iterativealgebras,streamcircuits,bms13} are also instances of the locally finite fixpoint, because in all those cases the fp and fg objects coincide. For example, the class of regular languages is the rational fixpoint of $2×(-)^\Sigma$ on \Set. In this
section, we will study further instances of the LFF that are most likely not instances of the rational fixpoint and which -- to the best of
our knowledge -- have not been characterized by a universal property yet:
\begin{enumerate}
\item Behaviours of finite-state machines with side-effects as considered by
the generalized powerset construction (cf.~Section~4.1),%\ref{sec:powerset}),
particularly the following.
\begin{enumerate}
\item Deterministic and ordinary context-free languages obtained as the behaviours of deterministic and non-deterministic stack-machines, respectively.
\item Constructively $S$-algebraic formal power series, i.e.~the
\textqt{context-free} subclass of weighted languages with weights from a
semiring $S$, yielded from weighted context-free grammars.
\end{enumerate}
\item The monad of Courcelle's algebraic trees.
\vspace{-4mm}
\end{enumerate}

%\subsection{Generalized Powerset Construction}
%\label{sec:powerset}
\densesubsection{Generalized Powerset Construction}
The determinization of a non-deterministic automaton using the powerset
construction is an instance of a more general framework, described by
\citet*{generalizeddeterminization} based on an observation by \citet{bartelsphd} (see also \citet{Jacobs05abialgebraic}). 
In that \emph{generalized powerset
construction}, an automaton with side-effects is turned into an ordinary
automaton by internalizing the side-effects in the states. The LFF
interacts well
with this construction, because it precisely captures the behaviours of
finite-state automata with side effects. The notion of side-effect is formalized
by a monad, which induces the category, in which the LFF is considered.

\begin{notes}
\begin{itemize}
  \item[] Some thoughts:
  \item conceptually, one starts with $X \to HTX$ with $T$ a monad,
  and obtains $FX \to H^T (FX)$, that is an coalgebra with free
  carrier. This corresponds to shifting focus from
  $\mathsf{Set}$-coalgebras to coalgebras over the Eilenberg-Moore
  category induced by $T$.
  \item In particular, in EM categories fp and fg don't generally
  agree.
  \item We show that the LFFP captures \emph{precisely} the
  behaviour of determinized EM-coalgebras, that is, of all
  coalgebras with free carrier.
\end{itemize}
\end{notes}

In the following we assume that readers are familiar with monads and Eilenberg-Moore algebras (see e.g.~\cite{lane1998categories} for an introduction). For a monad $T$ on $\C$ we denote by $\C^T$ the category of Eilenberg-Moore algebras. Recall from~\cite[Corollary~2.75]{adamek1994locally} that if $\C$ is lfp (in most of our examples $\C$ is \Set) and $T$ is finitary then $\C^T$ is lfp, too, and for every fp object $X$ the free Eilenberg-Moore algebra $TX$ is fp in $\C^T$. In all the examples we consider below, the classes of fp and fg objects either provably differ or it is still unknown whether these classes coincide.
\vspace{-3mm}

\takeout{ % taken out to safe space
\paragraph{Monads.}Recall that a (finitary)
\emph{monad} on $\C$ is a (finitary) endofunctor $T: \C\to \C$ with two natural
transformations $\eta: \Id\to T$, $\mu: T^2 \to T$ such that $\mu\cdot \mu T =
\mu \cdot T\mu$ and $\id_T = \mu\cdot \eta T = \mu \cdot T\eta$.

\noindent\partitle{Algebras for a monad.}
Each monad induces the so called \emph{Eilenberg-Moore category $\C^T$}, having
as objects the \emph{Eilenberg-Moore algebras}, i.e.~algebras $(A, \alpha: TA\to A)$
fulfilling $\alpha\cdot \eta_A = \id_A$ and $\alpha \cdot \mu_A = \alpha\cdot
T\alpha$, together with $T$-algebra homomorphisms as morphisms. The functor $F:
\C\to \C^T$ sending objects $X$ to free algebras $(TX, \mu_X)$ and the
forgetful functor form an adjoint pair $F: \smash{\adjointpair} :U$.
If $\C$ is lfp (mostly in our cases, $\C$ is \Set) and $T$ is finitary, then
$\C^T$ is lfp and $F$ preserves fp~objects
\cite[Corollary~2.75]{adamek1994locally}. In the examples we consider, the classes of fp and fg
objects either provably differ or it is still unknown whether these classes coincide.
} % end takeout

\begin{example} \label{exceptionmonad}
  In Sections~4.4 %\ref{sec:cfl} 
and~4.5 %\ref{sec:alg} 
we are going to make use of Moggi's exception monad transformer (see e.g.~\cite{cm93}). Let us recall that for a fixed object $E$, the finitary functor $(-)+E$ together with the unit
    $\eta_X = \inl: X\to X+E$ and multiplication $\mu_X = \id_X+[\id_E,\id_E]:
    X+E+E \to X+E$ form a finitary monad, the \emph{exception monad}. Its
    algebras are $E$-pointed objects, i.e.~objects $X$, together with a morphism $E\to X$, and homomorphisms are morphisms preserving the pointing. So the induced Eilenberg-Moore category is just the slice category $\C^{(-)+E} \cong E/\C$.
    
    Now, given any monad $T$ we obtain a new monad $T(- + E)$ with obvious unit and multiplication. An Eilenberg-Moore algebra for $T(-+E)$ consists of an Eilenberg-Moore algebra for $T$ and an $E$-pointing, and homomorphisms are $T$-algebra homomorphisms preserving the pointing~\cite{combiningeffects}.  
\end{example}
\vspace{-2mm}
Now an automaton with side-effects is modelled as an $HT$-coalgebra, where $T$ is a finitary monad on $\C$ providing the type of side-effect. For example, for $HX = 2 \times X^\Sigma$, where $\Sigma$ is an input alphabet, $2 = \{0,1\}$ and $T$ the finite powerset monad on \Set, $HT$-coalgebras are non-deterministic automata. However, the coalgebraic semantics using the final $HT$-coalgebra does not yield the usual language semantics of non-deterministic automata. To obtain this one considers the final coalgebra of a lifting of $H$ to $\C^T$. Denote by $U: \C^T \to \C$ the canonical forgetful functor.
\begin{definition}\label{deflifting}
    For a functor $H: \C\to \C$ and a monad $T: \C \to \C$,
    a \emph{lifting} of $H$ is a functor $H^T: \C^T\to \C^T$ such that
    $H\cdot U = U \cdot H^T$.
\end{definition}
If such a (not necessarily unique) lifting exists, the generalized powerset
construction transforms an $HT$-coalgebra into a $H^T$-coalgebra on $\C^T$:
For a coalgebra $x: X\to HTX$, $HTX$ carries an Eilenberg-Moore algebra, and one
uses freeness of the Eilenberg-Moore algebra $TX$ to obtain a canonical
$T$-algebra homomorphism $x^\sharp: (TX,\mu^T) \to H^T(TX,\mu^T)$. The
\emph{coalgebraic language semantics} of $(X,x)$ is then given by $X
\xrightarrow{\eta_X} TX \smash{\xrightarrow{x^{\sharp\dagger}}} \nu H^T$, i.e.~by composing the unique coalgebra morphism induced by $x^\sharp$ with $\eta_X$.  
%\[
%    \begin{tikzcd}
%        X
%        \arrow{r}{\eta_X}
%        \arrow{d}[left]{x}
%        &
%        TX
%        \arrow{r}{x^{\sharp\dagger}}
%        \arrow[dashed]{dl}{x^\sharp}
%        &
%        \nu H
%        \arrow{d}{\tau}
%        \\
%        HTX
%        \arrow{rr}{Hx^{\sharp\dagger}}
%        &&
%        H \nu H
%    \end{tikzcd}
%\]
This construction yields a functor
%This principle corresponds to shifting focus from $\Set$-coalgebras to
%coalgebras in the Eilenberg-Moore category of a monad $T$ and can be interpreted
%as a functor
\(
     T': \Coalg (HT) \to \Coalg H^T
\)
mapping coalgebras $X\xrightarrow{x} HTX$ to $x^\sharp$ and homomorphisms $f$ to
$Tf$ (see e.g.~\cite[Proof of Lemma 3.27]{bms13} for a proof).

Now our aim is to show that the LFF of $H^T$ characterizes precisely the coalgebraic language semantics of all fp-carried
$HT$-coalgebras. As the right adjoint $U$ preserves monos and is faithful, we
know that $H^T$ preserves monos, and as $T$ is finitary, filtered colimits in
$\C^T$ are created by the forgetful functor to $\C$, and we therefore see that $H^T$ is finitary. Thus, by Theorem~\ref{thm:final}, 
$\LFF H^T$ exists and is a subcoalgebra of $\nu H^T$. By \cite{plotkinturi97} and \cite[Corollary~3.4.19]{bartelsphd}, we know that $\nu H^T$ is carried by $\nu H$
equipped with a canonical algebra structure. %, see \cite[Notation 3.22]{bms13}.

Now let us turn to the desired characterization of $\LFF H^T$. Formally, the coalgebraic language semantics of all fp-carried $HT$-coalgebras is collected by forming the colimit $k: K \to HK$ of the diagram
\(
    %D = \big(
    \Coalgfg HT \xrightarrow{T'} \Coalg H^T \xrightarrow{U} \Coalg H.
    %\big).
\)
%Denote its colimit by $k: K\to HK$ with injections $\inj_X: (TX,x^\sharp)\to (K,k)$. % injections not used below
%However, not all identical behaviours are necessarily identified in $K$.
%But identifying those gives precisely the locally finite fixpoint of $H^T$ 
This coalgebra $K$ is not yet a subcoalgebra of $\nu H$ (for $\C = \Set$ that
means, not all behaviourally equivalent states are identified in $K$), but taking its image in $\nu H$ we obtain the LFF:
\begin{proposition} \label{firstLFFImage}
    The image $(I,i)$ of the unique coalgebra morphism $k^\dagger: K \to \nu H^T$ is precisely the locally finite fixpoint
    of the lifting $H^T$.% , i.e.~$(I,i) \cong (U\LFF H^T, U\ell)$. -- würde ich weglassen, die U's machen die Sache eher komplizierter...
\end{proposition}
%So when combining all behaviours in $K$ and then identifying behaviours gives the LFF.
%
%\enlargethispage{5pt}
One can also directly take the union of all desired behaviours, for $\C = \Set$:
%, because -- roughly speaking -- epis commute with colimits:
\begin{theorem} \label{prop:LFFunion}
    The locally finite fixpoint of the lifting $H^T$ comprises precisely the
    images of determinized $HT$-coalgebras:
\begin{equation}
\label{LFFunion}
    \LFF H^T = \ \bigcup_{\mathclap{\substack{x: X\to HTX\\
                           X\textrm{ \upshape finite}}}}
        \ x^{\sharp\dagger}[TX]
    = \ \bigcup_{\mathclap{\substack{x: X\to HTX\\
                           X\textrm{ \upshape finite}}}}
        \ x^{\sharp\dagger}\cdot\eta_X^T[X]
        \subseteq \nu H^T.
\end{equation}
\end{theorem}
This result suggests that the locally finite fixpoint is the right
object to consider in order to represent finite behaviour. We now
instantiate the general theory with examples from the literature to
characterize several well-known notions as LFF.

%\subsection{The Languages of Non-deterministic Automata}
\densesubsection{The Languages of Non-deterministic Automata}
Let us start with a simple standard example. We already mentioned that non-de\-ter\-mi\-nis\-tic automata are coalgebras for the functor $X \mapsto 2 \times \Potf(X)^\Sigma$. Hence they are $HT$-coalgebras for $H = 2 \times (-)^\Sigma$ and $T = \Potf$ the finite powerset monad on $\Set$. The above generalized powerset construction then instantiates as the usual powerset construction that assigns to a given non-deterministic automaton its determinization. 

Now note that the final coalgebra for $H$ is carried by the set $\mathcal L = \Pot(\Sigma^*)$ of all formal languages over $\Sigma$ with the coalgebra structure given by $o: \mathcal L \to 2$ with $o(L) = 1$ iff $L$ contains the empty word and $t: \mathcal L \to \mathcal{L}^\Sigma$ with $t(L)(s) = \{ w \mid sw \in L\}$ the left language derivative. The functor $H$ has a canonical lifting $H^T$ on the Eilenberg-Moore category of $\Potf$, viz.~the category of join semi-lattices. The final coalgebra $\nu H^T$ is carried by all formal languages with the join semi-lattice structure given by union and $\emptyset$ and with the above coalgebra structure. Furthermore, the coalgebraic language semantics of $x: X \to HTX$ assigns to every state of the non-deterministic automaton $X$ the language it accepts. Observe that join semi-lattices form a so-called \emph{locally finite variety}, i.e.~the finitely presentable algebras are precisely the finite ones. Hence, Theorem~\ref{prop:LFFunion} states that the LFF of $H^T$ is precisely the subcoalgebra of $\nu H^T$ formed by all languages accepted by finite NFA, i.e.~regular languages. 

Note that in this example the LFF and the rational fixpoint coincide since both fp and fg join semi-lattices are simply the finite ones. Similar characterizations of the coalgebraic language semantics of finite coalgebras follow from Theorem~\ref{prop:LFFunion} in other instances of the generalized powerset construction from~\cite{generalizeddeterminization} (cf.~e.g.~the treatment of the behaviour of finite weighted automata in~\cite{bms13}).  

We now turn to examples that, to the best of our knowledge, cannot be treated using the rational fixpoint. 
 
%\subsection{The Behaviour of Stack Machines}
\densesubsection{The Behaviour of Stack Machines}
Push-down automata are finite state machines with infinitely many
configurations.
It is well-known that deterministic and non-deterministic
pushdown automata recognize different classes of context-free languages. We will
characterize both as instances of the locally finite fixpoint, using the results from \cite{coalgchomsky}
on stack machines, which can push or read
\emph{multiple} elements to or from the stack in a single transition, respectively.

That is, a transition of a stack machine in a certain state consists of reading
an input character, going to a successor state based on the stack's topmost
elements and of modifying the topmost elements of the stack. These stack
operations are captured by the stack monad.

\begin{definition}[{Stack monad, \cite[Proposition 5]{goncharov13}}]
    For a finite set of stack symbols $\Gamma$, the \emph{stack monad} is the
    submonad $T$ of the store monad $(-×\Gamma^*)^{\Gamma^*}$ for which the
    elements $\fpair{r,t}$ of $TX \subseteq (X×\Gamma^*)^{\Gamma^*}
    \cong X^{\Gamma^*} \times (\Gamma^*)^{\Gamma^*}\!$ satisfy the
    following restriction: there exists $k$ depending on $r,t$ such that for
    every $w\in \Gamma^k$ and $u\in \Gamma^*$, $r(wu) =r(w)$ and $t(wu) =
    t(w)u$.
    \vspace{-2mm}
\end{definition}

\noindent
Note that the parameter $k$ gives a bound on how may of the topmost stack cells the machine can
access in one step. 

Using the stack monad, stack machines are $HT$-coalgebras, where $H=B×(-)^\Sigma$ is the
Moore automata functor for the finite input alphabet $\Sigma$ and the set $B$ of
all predicates mapping (initial) stack configurations to output values from 2,
taking only the topmost $k$ elements into account:
\(
    B = \{
        p \in 2^{\Gamma^*}
        \mid
        \exists k\in \N_0: 
        \forall w,u\in \Gamma^*, |w|\ge k:
        p(wu) = p(w)
    \} \subseteq 2^{\Gamma^*}
\).

\noindent
The final coalgebra $\nu H$ is carried by $B^{\Sigma^*}$ which is
(modulo power laws) a set of predicates, mapping stack configurations to formal
languages. \citet{coalgchomsky} show that $H$ lifts to $\Set^T$ and
conclude that
finite-state $HT$-coalgebras match the intuition of \emph{deterministic}
pushdown automata without spontaneous transitions. The languages
accepted by those automata are precisely the \emph{real-time deterministic context-free
languages}; this notion goes back to \citet{HarrisonHavel72}.
We obtain the following, with $\gamma_0$ playing
the role of an initial symbol on the stack:

\begin{theorem}
    The locally finite fixpoint of $H^T$ is carried by the set of all maps $f \in B^{\Sigma^*}$ such that for any fixed $\gamma_0 \in \Gamma$, 
    \(
        \{ w\in \Sigma^* \mid f(w)(\gamma_0) = 1\}
    \)
    is a real-time deterministic context-free language.
\end{theorem}
\begin{proof}
    By \cite[Theorem 5.5]{coalgchomsky}, a language $L$ is a real-time
    deterministic context-free language iff there exists some $x: X\to
    HTX$, $X$ finite, with its determinization $x^\sharp: TX \to HTX$ and there
    exist $s\in X$ and $\gamma_0 \in \Gamma$ such that $f =
    x^{\sharp\dagger}\cdot\eta_X^T(s) \in B^{\Sigma^*}$ and $f(w)(\gamma_0) = 1$ for
    all $w\in \Sigma^*$. The rest follows by \eqref{LFFunion}.
    \qed
\end{proof}
Just as for pushdown automata, the expressiveness of stack machines increases when
equipping them with non-determinism. Technically, this is done by considering
the \emph{non-deterministic stack monad} $T'$, i.e.~$T'$ denotes a submonad of the
non-deterministic store monad $\Potf(-×\Gamma^*)^{\Gamma^*}$, as described in
\cite[Section 6]{coalgchomsky}. In the non-deterministic setting, a similar property
holds, namely that the determinized $HT'$-coalgebras with finite carrier
describe precisely the context-free languages \cite[Theorem 6.5]{coalgchomsky}.
Combine this with \eqref{LFFunion}:
%\vspace{-2mm}
\begin{theorem}
    The locally finite fixpoint of $H^{T'}$ is carried by the set of all maps $f \in B^{\Sigma^*}$ such that for any fixed $\gamma_0 \in \Gamma$,
    \(
        \{ w\in \Sigma^* \mid f(w)(\gamma_0) = 1\}
    \)
    is a context-free language.
\end{theorem}
%\vspace{-4mm}

%\subsection{Context-Free Languages and Constructively $S$-Algebraic %Formal 
%Power Series}
%\label{sec:cfl}
\begin{notes}
\begin{itemize}
  \item[] Some ideas:
  \item weighted languages generalize ordinary languages and are
  construed as functions $\Sigma^* \to S$, with $S$ a semiring
  \item of particular interest: $S$-algebraic formal power series
  that generalize context-free languages
  \item in our terminology, they are generated by coalgebras $X \to
  S \times (S\langle X \rangle)^\Sigma$ where $X$ is finite.
  \item We show that the set of constructively $S$-algebraic
  languages arises precisely as the LFFP for the (standard) functor
  for deterministic automata, i.e. $X \to S \times X^\Sigma$, but
  in the category of Eilenberg-Moore algebras for the monad
  $S\langle \cdot \rangle$.
\end{itemize}
\end{notes}

\densesubsection{Context-Free Languages and Constructively $S$-Algebraic Power Series}
One generalizes from formal (resp.~context-free) languages to weighted formal
(resp.\ context-free) languages by assigning to each word a weight from a fixed
semiring. More formally, a weighted language -- a.k.a.~\emph{formal power series}
-- over an input alphabet $X$ is defined as a map $X^* \to S$, where $S$ is a
semiring. The set of all formal power series is denoted by $\FPS{X}$. Ordinary
formal languages are formal power series over the boolean semiring $\mathbb{B}=
\{0,1\}$, i.e.~maps $X^*\to \{0,1\}$.

An important class of formal power series is that of \emph{constructively
$S$-algebraic} formal power series. We show that this class arises precisely as
the LFF of the standard functor for deterministic Moore automata $H = S×(-)^\Sigma$, but
on an Eilenberg-Moore category of a \Set monad. As a special case, constructively
$\mathbb{B}$-algebraic series are the context-free weighted languages and are precisely
the LFF of the automata functor in the category of idempotent semirings.

The original definition of constructively $S$-algebraic formal power series goes back to
\citet{Fliess1974}, see also~\cite{handbookWeightedAutomata}.
Here, we use the equivalent coalgebraic characterization by \citet{jcssContextFree}.

Let $\Poly{X}\subseteq \FPS{X}$ the subset of those maps, that are $0$ for
all but finitely many $w\in X^*$. If $S$ is commutative, then $\Poly{-}$ yields a finitary
monad and thus also $T=\Poly{-+\Sigma}$ by~\autoref{exceptionmonad}. %\cite[Corollary 1]{combiningeffects}.
The algebras for $\Poly{-}$ are associative $S$-algebras (over the commutative
semiring $S$), i.e.~$S$-modules together with a monoid structure that is a module morphism in
both arguments. The algebras for $T$ are $\Sigma$-pointed $S$-algebras. %\cite{combiningeffects}. 
The following notions are special instances of $S$-algebras.
\vspace{-1mm}
\begin{example}
    For $S=\mathbb{B} = \{0,1\}$, one obtains idempotent semirings as
    $\mathbb{B}$-algebras, for $S=\mathbb{N}$ semirings, and for
    $S=\mathbb{Z}$ ordinary rings.
\end{example}
\vspace{-1mm}

\citet[Proposition 4]{jcssContextFree} show that the final $H$-coalgebra is carried by
$\FPS{\Sigma}$ and that constructively $S$-algebraic series are precisely those
elements of $\FPS{\Sigma}$ that arise as the behaviours of those coalgebra $c: X\to H\Poly{X}$ with
finite $X$, after determinizing them to some $c^\sharp: \Poly{X} \to H\Poly{X}$ (see~\cite[Theorem 23]{jcssContextFree}).

However, this determinization is not directly an instance of the generalized
powerset construction. We shall show that the same behaviours can be obtained by using the standard generalized powerset construction 
with an appropriate lifting of $H$ to $T$-algebras.
Having an $S$-algebra structure on $A$ and a $\Sigma$-pointing $j: \Sigma \to A$ we need to
define another $S$-algebra structure and $\Sigma$-pointing on $HA = S×A^\Sigma$. While the $S$-module structure is just
point-wise, we need to take care when multiplying two elements from $HA$. To this end we first we define the operation $\fuse{-,-}: S \times A^\Sigma \to A$ by
\vspace{-2mm}
\[
    \fuse{o,\delta} :=
        i(o) + \displaystyle\sum_{b\in \Sigma}\big(j(b)\cdot \delta(b)\big),
\vspace{-1mm}
\]
where $i: S \to A$ is the canonical map with $i(s) = s \cdot 1$ with $1$ the neutral element of the monoid on $A$. 
The idea is that $\fuse{o,\delta}$ acts like a state with output $o$ and
derivation $\delta$. The multiplication on $HA = S \times A^\Sigma$ is then defined by
\vspace{-1mm}
\begin{align}
    \label{Hmultiplication}
    (o_1,\delta_1)* (o_2,\delta_2) :=
    \big(o_1\cdot o_2, a\mapsto \delta_1(a)\cdot \fuse{o_2,\delta_2} +
    i(o_1)\cdot \delta_2(a) \big).
\end{align}
The $\Sigma$-pointing is the obvious: $a \mapsto (0, \varrho_a)$ where
$\varrho_a(a) = 1$ and $\varrho_a(b) = 0$ 
for $a \neq b$.
\begin{lemma}
    \label{bisimilarAlgebra}
    For any $w \in A$ in $\Set^T$ and any $H^T$-coalgebra structure
    $c: A\to H^TA$, $w$ and $\fuse{c(w)}$ are behaviourally equivalent in \Set.
\end{lemma}
Given a coalgebra $c: X\to H\Poly{X}$, \citet[Proposition~14]{jcssContextFree} determinize
$c$ to some $\hat c = \fpair{\hat o,\hat\delta}: \Poly{X}\to H\Poly{X}$ with the
property that for any $v,w\in \Poly{X}$,
\vspace{-1mm}
\begin{align}
    \label{WEAproperty}
    \hat o(v*w) = \hat o(v) \cdot \hat o(w)
    \quad\text{and}\quad
    \hat\delta(v*w,a) = \hat\delta(v,a)*w+\hat o(v)*\hat\delta(w,a),
\end{align}
and such that $\hat c$ is a $S$-module homomorphism. However, the generalized powerset construction w.r.t.~$T$ yields a coalgebra $c^\sharp: \Poly{X+\Sigma} \to H\Poly{X + \Sigma}$. The above property, together with
\autoref{bisimilarAlgebra} and \eqref{Hmultiplication} implies that
$\hat c$ and $c^\sharp$ are essentially the same coalgebra structures:
\begin{lemma}
    \label{CFSameCoalgebra}
    In \Set, $u\in (\Poly{X},\hat c)$ and $\Poly{\inl}(u) \in (\Poly{X+\Sigma}, c^\sharp)$ are behaviourally equivalent.
\end{lemma}
It follows that $\hat c^\dagger = c^{\sharp\dagger}\cdot \Poly{\inl}$ and thus their
images in $\nu H$ are identical. Hence, a formal power series is constructively
$S$-algebraic iff it is in the image of some $c^{\sharp\dagger}\cdot
\Poly{\inl}$, and by \eqref{LFFunion}, iff it is in the locally finite fixpoint
of $H^T$.

%\newpage
%\subsection{Courcelle's Algebraic Trees}
%\label{sec:alg}
\begin{notes}
\begin{itemize}
  \item What are Courcelle's algebraic trees?
  \item Why are they important?
  \item They arise as the locally finite fixpoint.
\end{itemize}
\end{notes}
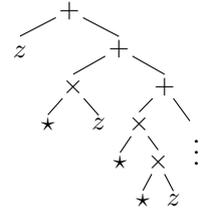
\begin{wrapfigure}{r}{3cm} % alignment, width
    \centering
    % we really need these vspaces, see
    % https://en.wikibooks.org/wiki/LaTeX/Floats,_Figures_and_Captions#Tip_for_figures_with_too_much_white_space
    %\vspace{-1cm}
    %\vspace{-10pt}
    \begin{tikzpicture}[lambdatree,mathnodes, level distance=5mm]
    \node (z) {+ }
    child { node {\mathllap{\phantom{X}}z} }
    child { node[yshift=0mm] {+}
            child { node[xshift=-1mm] {×}
                child { node {\star} }
                child { node {z} } }
            child { node[xshift=1mm] {+}
                child { node {×}
                    child { node {\star} }
                    child { node {×}
                        child { node {\star} }
                        child { node {z} } } }
                child { node (ddots) {\ } }
            }
            }
    ;
            \node[anchor=north west] at (ddots.north) {\vdots};
\end{tikzpicture}
    \vspace{-3mm}
    \caption{Solution of \newline\hspace*{\fill} $\varphi(z) = z + \varphi(\star × z)$}
    \label{rpsexample}
    \vspace{-22pt}
\end{wrapfigure}

\densesubsection{Courcelle's Algebraic Trees}
For a fixed signature $\Sigma$ of so called \emph{givens}, a \emph{recursive
program scheme} (or \emph{rps}, for short) contains mutually recursive definitions of
new operations $\varphi_1,\ldots,\varphi_k$ (with respective arities
$n_1,\ldots,n_k$). The recursive definition of $\varphi_i$ may involve
symbols from $\Sigma$, operations $\varphi_1,\ldots,\varphi_k$ and
$n_i$ variables $x_1,\ldots,x_{n_i}$. The (uninterpreted) solution of an rps is
obtained by unravelling these recursive definitions, producing a possibly
infinite $\Sigma$-tree over $x_1,\ldots,x_{n_i}$ for each operation $\varphi_i$.
\autoref{rpsexample} shows an rps over the signature $\Sigma
=\{\nicefrac{\star}{0}, \nicefrac{×}{2},
\nicefrac{+}{2}\}$ and its solution. In general, an \emph{algebraic $\Sigma$-tree} is a $\Sigma$-tree which is definable by an
rps over $\Sigma$ (see Courcelle~\cite{courcelle}). Generalizing from a signature to a finitary endofunctor
$H:\C\to \C$ on an lfp category, \citet{secondordermonad} describe
an rps as a coalgebra for a functor $\Hf$ on $H/\Mndf(\C)$, in which objects are finitary $H$-pointed monads on $\C$, 
i.e.~finitary monads $M$ together with a natural transformation $H\to M$. They introduce the \emph{context-free} monad $C^H$ of $H$, which is  an $H$-pointed monad that is a subcoalgebra of the final coalgebra for $\Hf$ and which is the monad of Courcelle's algebraic $\Sigma$-trees in the special case where $\C = \Set$ and $H$ is a polynomial functor associated to a signature $\Sigma$. We will prove that this monad is the LFF of $\Hf$, and thereby we characterize it by a universal property -- solving the open problem in~\cite{secondordermonad}. 

The setting is again an instance of the generalized powerset construction, but this time with $\Funf(\C)$ as the base category in lieu of $\Set$. Let $\C$ be an lfp category in
which the coproduct injections are monic and consider a finitary,
mono-preserving endofunctor $H: \C\to \C$. Denote by $\Funf(\C)$ the category of
finitary endofunctors on $\C$. Then $H$ induces an endofunctor $H\cdot(-)+\Id$
on $\Funf(\C)$, denoted $\dot H$ and mapping an endofunctor $V$ to the functor
$X\mapsto HVX+X$.
This functor $\dot H$ gets precomposed with a monad on $\Funf(\C)$ as we now explain.  
%% CONTINUE HERE
%\vspace{-1mm}
\begin{proposition}[Free monad, \cite{freealgebras,freetriples}]\label{prop:mon}
    For a finitary endofunctor $H$, free $H$-algebras $\varphi_X: HF^HX\to F^HX$
    exist for all $X\in \C$. $F^H$ itself is a finitary monad on
    $\C$, more specifically it is the \emph{free monad} on $H$.
\end{proposition}
%\vspace{-2mm}
For example, if $H$ is a polynomial functor associated to a signature $\Sigma$, then $F^{H}X$ is the usual term algebra that contains all finite $\Sigma$-trees over the set of generators $X$. Proposition~\ref{prop:mon} implies that $H \mapsto F^H$ is the object assignment of a monad on $\Funf(\C)$. The Eilenberg-Moore category of $F^{(-)}$ is easily seen to be $\Mndf(\C)$, the category of finitary
monads on $\C$. %, together with monad morphisms.
Here, fp and fg objects differ, see \cite[Section~5.4.1]{wissmann2015} for a proof.

Similarly as in the case of context-free languages, we will work with the monad
$E^{(-)} = F^{H+(-)}$, so we get $H$-pointed finitary monads as the
$E^{(-)}$-algebras. This category is equivalent to a slice category: the universal
property induced by $F^{(-)}$ states, that for any finitary monad $B$ the
natural transformations $H\to B$ are in one-to-one correspondence
with monad
morphisms $F^H \to B$; so the category $H/\Mndf(\C)$ of finitary $H$-pointed monads on $\C$ 
is isomorphic to the slice category $F^H/\Mndf(\C)$.  This finishes the description of the base
category and we now lift the functor $\dot H$ to this category. 

Consider an $H$-pointed monad $(B,\beta: H\to \C) \in H/\Mndf(\C)$. By
\cite{monadsofcoalgebras}, the endofunctor $H\cdot B+\Id$ carries a canonical
monad structure. Furthermore, we have an obvious pointing $\inl\cdot H\eta^B:
H\to H\cdot B +\Id$. By \cite{mmcatsolrps}, this defines an endofunctor on
$H$-pointed monads, 
\(
    \Hf: H/\Mndf(\C)\to H/\Mndf(\C),
\)
which is a lifting of $\dot H$. In order to verify that $\Hf$ is finitary, we
first need to know how filtered colimits look in $H/\Mndf(\C)$.

\vspace{-2mm}
\begin{lemma}\label{mndfilteredcolimits}
    The forgetful $U: \Mndf(\C) \to \Funf(\C)$ creates filtered colimits.
\end{lemma}
\vspace{-2mm}
Clearly, the canonical projection functor $H/\Mndf(\C) \to \Mndf(\C)$ creates filtered colimits, too. Therefore, filtered colimits in the slice category $H/\Mndf(\C)$ are formed on the level of $\Funf(\C)$, i.e.~object-wise. The functor $\dot H$ is
finitary on $\Funf(\C)$ and thus also its lifting $\Hf$ is finitary.
So all requirements from \autoref{basicassumption} are met: we have a
finitary endofunctor $\Hf$ on the lfp category $H/\Mndf(\C)$, and by
\cite[Corollary 2.20]{secondordermonad} $\Hf$ preserves monos since $H$ does. By \autoref{thm:final}, $\Hf$ has a locally finite fixpoint.
\vspace{-2mm}

\begin{remark}\label{rem:T}
The final $\Hf$-coalgebra is not of much interest, but that of a related 
functor. $\Hf$ generalizes to a functor \( \H: H/\Mndc(\C) \to H/\Mndc(\C) \) on $H$-pointed countably
accessible\footnote{A colimit is \emph{countably filtered} if its diagram has for every
countable subcategory a cocone. A functor is \emph{countably accessible} if
it preserves countably filtered colimits.} monads. For any
object $X \in \C$, the finitary endofunctor $H(-) + X$ has a final coalgebra;
call the carrier $TX$. Then $T$ is a monad \cite{aamvcia}, is countably
accessible~\cite{secondordermonad} and is the final $\H$-coalgebra
\cite{mmcatsolrps}.
\end{remark}
\vspace{-2mm}
\noindent
\citet{secondordermonad} characterize a (guarded) recursive program scheme as 
a natural transformation $V \to H \cdot E^{V} + \Id$ with $V$ fp (in
$\Funf(\C)$), or equivalently, via the generalized powerset construction
w.r.t.~the monad $E^{(-)}$ as an $\Hf$-coalgebra on the carrier $E^V$ (in
$\Mndf(\C)$). These $\Hf$-coalgebras on carriers $E^{V}$ where $V\in \Funf(\C)$
is fp form the full subcategory $\EQ \subseteq \Coalg \Hf$.
They show two equivalent ways of constructing the monad of
Courcelle's algebraic trees for the case $\C=\Set$:
as the image of $\colim \EQ$ in the final coalgebra $T$ of~\autoref{rem:T}, and 
as the colimit of $\EQ[2]$, where $\EQ[2]$ is the closure of $\EQ$
under strong quotients.
%
%\begin{enumerate}
%\item as the image of $\colim \EQ$ in $T$.
%\item as the colimit of $\EQ[2]$, where $\EQ[2]$ is the closure of $\EQ$
%under coequalizers and strong quotients.
%\end{enumerate}
We now provide a third characterization, and show that the monad of
Courcelle's algebraic trees is the locally finite fixpoint of $\Hf$.
%\begin{enumerate}[resume]
%\item as the locally finite fixpoint of $\Hf$.
%\end{enumerate}

To this end it suffices to show that $\EQ[2]$ is precisely the diagram of
$\Hf$-coalgebras with an fg carrier. This is established with the help of the following two technical lemmas. We now assume that $\C = \Set$. 
\vspace{-2mm}
\begin{lemma}
    \label{quasiEpiPreservation}
    $\Hf$ maps strong epis to morphisms carried by strong epi natural
    transformations.
\end{lemma}
\vspace{-1mm}
We have the following variation of \autoref{lfpquotient}:
\vspace{-1mm}
%\vspace{-1mm}
\begin{lemma}
    \label{EQquotient}
    Any $\Hf$-coalgebra $b: (B,\beta) \to \Hf(B,\beta)$, with $B$ fg, is the
    strong quotient of a coalgebra from $\EQ$.
\end{lemma}
The proof of~\autoref{EQquotient} makes use of Lemma~\ref{quasiEpiPreservation}
as well as the following properties:
\vspace{-2mm}
\begin{itemize}
\item The fp objects in $\Funf(\Set)$ are the quotients of polynomial functors.
\item The polynomial functors are projective. That means that for a polynomial functor $P$ and any natural transformation $n: K \to L$ with surjective components we have the following property:  for every $f: P \to L$ there exists $f': P \to K$ with $n \cdot f' = f$. 
\item Any fg object in $H/\Mndf(\Set)$ is the quotient of some $E^V$ with $V$ fp in
$\Funf(\Set)$ and thus also of some $E^P$ with $P$ a polynomial functor.
\end{itemize}
\vspace*{-2mm}
Note that the last property holds because $H/\Mndf(\Set)$ is an Eilenberg-Moore category and $E^V$ is the free Eilenberg-Moore algebra on the fp object $V$. 
%A technical difficulty that arises is that for a projective $P$ in $\Funf(\Set)$, $E^P$
%is not necessarily projective in $H/\Mndf(\Set)$ (w.r.t.~monad morphisms with surjective components). However, $E^P$ is projective enough for our purposes: %%
%(for details see the appendix). 
%
%\begin{proof}
%    We know that $(B,\beta)$ is the strong quotient of a $M^V$ with $V$ fp in
%    $\Funf(\Set)$, and thus also of some $M^P$, with $P$ a polynomial functor:
%    $q: M^P \to (B,\beta)$. $\Hf$ maps $q$ to something with an underlying
%    strong epi natural transformation, and $M^P$ is projective for it, by the
%    adjunction of $M^{(-)}$. The rest follows by \autoref{lfpquotient}.
%\end{proof}
It follows from~\autoref{EQquotient} that $\Coalgfg \Hf$ is the same category as $\EQ[2]$; thus their colimits in $\Coalg \Hf$ are
isomorphic and we conclude:
\vspace{-2pt}
\begin{theorem}
The locally finite fixpoint of $\Hf: H_\Sigma/\Mndf(\Set)\to H_\Sigma/\Mndf(\Set)$ is
the monad of Courcelle's algebraic trees, sending a set to the algebraic
$\Sigma$-trees over it.
\vspace{-3mm}
\end{theorem}

\section{Conclusions and Future Work}\enlargethispage{10pt}
\vspace{-2mm}
\label{sec:con}
We have introduced the locally finite fixpoint of a finitary mono-preserving endofunctor on an lfp category. We proved that this fixpoint is characterized by two universal properties: it is the final lfg coalgebra and the initial \fgiterative algebra for the given endofunctor. And we have seen many instances where the LFF is the domain of behaviour of finite-state and finite-equation systems. In particular all previously known instances of the rational fixpoint are also instances of the LFF, and we have obtained a number of interesting further instances not captured by the rational fixpoint.  

On a more technical level, the LFF solves a problem that sometimes makes the
rational fixpoint hard to apply. The latter identifies behaviourally equivalent
states (i.e.~is a subcoalgebra of the final coalgebra) if the classes of fp and
fg objects coincide. This condition, however, may be false or unknown (and
sometimes non-trivial to establish) in a given lfp category. But the LFF always
identifies behaviourally equivalent states.  

%always exists for finitary
%mono-preserving endofunctors on lfp categories. And we have seen that the LFF is
%an applicable framework to talk about finite-state and finite-equation systems,
%because we inherit all known instances of the rational fixpoint and in addition
%obtain a lot more instances, in order to characterize well-established notions
%by universal properties.

There are a number of interesting topics for future work concerning the LFF. First, it should be interesting to obtain further instances of the LFF, e.g.~analyzing the behaviour of tape machines~\cite{coalgchomsky} may perhaps lead to a description of the recursively enumerable languages by the LFF. Second, syntactic descriptions of the LFF are of interest.  In works such as~\cite{brs_lmcs,bbrs_ic,bms13,myersphd} Kleene type theorems and axiomatizations of the behaviour of finite systems are studied. Completeness of an axiomatization is then established by proving that expressions modulo axioms form the rational fixpoint. It is an interesting question whether the theory of the LFF we presented here may be of help as a tool for syntactic descriptions and axiomatizations of further system types.  

As we have mentioned already the rational fixpoint is the starting point for the coalgebraic study of iterative and iteration theories. A similar path could be followed based on the LFF and this should lead to new coalgebraic iteration/recursion principles, in particular in instances such as context-free languages or constructively $S$-algebraic formal power series.

Another approach to more powerful recursive definition principles are abstract operational rules (see~\cite{Klin11} for an overview). It has been shown that certain rule formats define operations on the rational fixpoint~\cite{bmr12,mbmr13}, and it should be investigated whether a similar theory can be developed based on the LFF.

Finally, in the special setting of Eilenberg-Moore categories one could base the study of finite systems on \emph{free} finitely generated algebras (rather than all fp or all fg algebras). Does this give a third fixpoint capturing behaviour of finite state systems with side effects besides the rational fixpoint and the LFF? And what is then the relation between the three fixpoints? Also the parallelism in the technical development between rational fixpoint and LFF indicates that there should be a general theory that is parametric in a class of ``finite objects'' and that allows to obtain results about the rational fixpoint, the LFF and other possible ``finite behaviour domains'' as instances.

%However, there are still some notions from the Chomsky hierarchy left which
%have not been characterized by universal properties yet. \citet{coalgchomsky}
%characterize tape machines coalgebraically as an automata having two kinds of
%side effects. One is the modification of a tape and the other is a possible
%$\varepsilon$-transition to another state. While the former monad is finitary,
%its extension by $\varepsilon$-transitions is not. It is open how to 
%modify the setting in order to describe their behaviour -- the decidable
%languages -- by a universal property.

%\newpage

%Furthermore, the LFF solves the technical fg vs fp question a category in which
%we consider finite-state coalgebras -- however not by answering it, but by
%providing a uniform framework that does not rely on the answer and still
%collects precisely the finite behaviours.

%
% Bibliography
%
\bibliography{refs,all2,delta2}

\iffull
    \newpage
    \begin{appendix}

      \section{Omitted Proofs and Results}

\subsection*{Technical Lemmas for \autoref{prop:lfgcolim}}

We first show directed unions of fg-carried coalgebras are lfg.
\begin{lemma}
    \label{unionoffgcoalgebras}
    Every directed union of coalgebras from $\Coalgfg H$ is an lfg coalgebra.
\end{lemma}
\begin{proof}
    Let $D: (I,\le) \to \Coalg H, (D_i,d_i) := Di$ be a directed diagram of
    coalgebras from $\Coalgfg H$ and of mono-carried morphisms. Name the colimit
    cocone $c_i: (D_i,d_i) \to (A,a)$ in $\Coalg H$. To check
    \autoref{lfgcoalgebra}, let $S$ be a finitely generated object with $f:
    S \to A$ in $\C$. As colimits in $\Coalg H$ are created by the forgetful
    functor $U: \Coalg H \to \C$, and because $U\cdot D$ is a directed diagram
    of monos and $S$ is an fg object, we obtain some factorization as shown below:
    \[
    \begin{tikzcd}[ampersand replacement=\&]
        S \arrow{rr}{f}
          \arrow{dr}[below left]{f'}
         \& \& U(A,a) = A
        \\
        \& UDi = D_i \arrow{ur}[below right]{Uc_i}
    \end{tikzcd}
    \]
    Note that because $U$ creates the colimits, we know that the colimit
    injection for $UDi$ in $\C$ is precisely $Uc_i$. \qed
\end{proof}

Next follow two easy technical lemmas on directed colimits. 
\begin{lemma}\label{colimitepi}
    For a directed diagram $D: \D\to \C$ of subobjects $m_i: C_i\rightarrowtail
    C$ of $C$, the colimit $(d_i: C_i\to \colim D)_{i\in \D}$ is obtained by
    taking the (strong epi,mono)-factorization of $\coprod C_i
    \xrightarrow{[m_i]} C$.
\end{lemma}
\begin{proof}
    At first, the $(m_i)_{i\in \D}$ form a cocone, so we have a unique $m:
    \colim D \to C$ with ${m\cdot d_i = m_i}$, and $d_i$ is monic. As $\C$ is
    lfp and both $d_i$ and $m_i$ are monic, \cite[Proposition 1.62
    (ii)]{adamek1994locally} gives us that $m$ is monic, too. The copair of a
    family of jointly strongly epic family $[d_i]:
    \coprod C_i\to \colim D$ is a strong epi and therefore we have the
    factorization:
    \[
    \begin{tikzcd}[anchor=base]
        \coprod C_i \arrow{rr}{[m_i]}
        \arrow[->>]{dr}[below left]{[d_i]}
        \PMNC \PMNC C
        \\
        \PMNC \colim D
        \arrow[>->]{ur}[below right]{m}
    \end{tikzcd}
    \tag*{\qed}
    \]
\end{proof}
\begin{lemma}\label{unionsofimages}
    Images of colimits in $\Coalg H$ are directed unions of images. More
    precisely, for a diagram $D: \D\to \Coalg H$, given a colimit cocone $(c_i:
    Di \to C)_{i\in \D}$ and a morphism $f: C\to B$, define $A_i$ as $\Im(f\cdot
    c_i)$. Then $\Im(f)$ is the directed union of the $A_i$ together with the
    induced monomorphisms:
    \begin{equation} \label{equnionofimages}
        \begin{tikzcd}
            Di \arrow[->>]{rr}{e_i} \arrow{d}[left]{c_i} & {} &
            A_i \arrow[dashed]{dl}[above left]{d_i} 
            \arrow[>->]{d}[right]{m_i} \\
            C \arrow[->>]{r}{e}
            \arrow[shiftarr={yshift=-3ex}]{rr}[below]{f}
            \arrow[->>]{r}{e}
            & \Im(f) \arrow[>->]{r}{m} & B
        \end{tikzcd}
    \end{equation}
\end{lemma}
\begin{proof}
    As colimits in $\Coalg H$ are created by the forgetful functor $U: \Coalg H
    \to \C$, we consider only the objects first. Take the (strong
    epi-carried,mono-carried)-factorizations $f\cdot c_i = m_i\cdot e_i$ for
    each $i\in \D$, and $f= m\cdot e$. Then \eqref{equnionofimages}
    where $d_i$ is induced by the strong epi $e_i$. Notice that by $m\cdot d_i =
    m_i$, $d_i$ is a mono as well.
    For any morphism $g: Di\to Dj$ we get a mono in $\bar g: A_i\rightarrowtail A_j$ by the
    strong epi $e_i$:
    \[
        \begin{tikzcd}[ampersand replacement=\&]
            Di\arrow[->>]{r}{e_i}
            \arrow{d}[left]{g}
            \& A_i \arrow[>->]{dr}[above right]{d_i}
            \arrow[dashed,>->]{d}[left]{\bar g}
            \\
            Dj\arrow[->>]{r}{e_j}
            \&
            A_j \arrow[>->]{r}[below]{d_j}
            \&
            \Im(f)
        \end{tikzcd}
    \]
    By $d_j\cdot\bar g = d_i$, we know that $\bar g$ is a mono as well. The
    $d_i$ also ensure that between each pair of objects $A_i,A_j$ there is at
    most one morphism. With this relation to the $D_i$, we also inherit the
    existence of upper bounds in $A_i$, which can be summarized in: the $A_i$
    form a directed diagram of monos in $\C$, i.e. a directed union in $\Coalg
    H$.

    To see that $\Im(f)$ is indeed its colimit, consider
    \[
        \begin{tikzcd}[ampersand replacement=\&]
            \coprod_{i} Di
            \arrow[->>]{r}{[c_i]}
            \arrow[->>]{d}[left]{\coprod e_i}
            \&
            C
            \arrow[->>]{d}{e}
            \\
            \coprod A_i
            \arrow{r}[below]{[d_i]}
            \&
            \Im(f)
        \end{tikzcd}
    \]
    which commutes, because \eqref{equnionofimages} did for every $i\in \D$.
    The copair of strong epis $[c_i]$ itself is a strong epi and so $e\cdot [c_i]$ and
    $[d_i]\cdot \coprod e_i$ as well. So
    $[d_i]$ is a strong epi and $[m_i]$ factors into $m$ and $[d_i]$, and by
    \autoref{colimitepi} $\Im(f)$, is the colimit.
    \[
    \begin{tikzcd}[ampersand replacement=\&]
        \coprod A_i \arrow[->>]{r}{[d_i]}
        \arrow[shiftarr={yshift=5mm}]{rr}[above]{[m_i]}
        \&
        \Im(f) \arrow[>->]{r}{m}
        \&
        B
    \end{tikzcd}
        \tag*{\qed}
    \]
\end{proof}

\subsection*{Proof of \autoref{prop:lfgcolim}}
\begin{proof}
    Let $c_i: (X_i,x_i) \to (X,x)$ be a
    colimit cocone of a filtered diagram with $(X_i,x_i)$ from $\Coalgfg H$.
    Take the (strong epi,mono)-factorizations
    \[
        c_i \equiv (
        \begin{tikzcd}
        X_i \arrow[->>]{r}{e_i} &
        T_i \arrow[>->]{r}{m_i} &
        X
        \end{tikzcd}
        )
    \]
    to get the subcoalgebras $(T_i,t_i)$ of $(X,x)$. By \autoref{unionsofimages}
    with $f=\id_X:X\to X$, $\Im(f) = X$ is the directed union of the $T_i$.
    These $T_i$ are in $\Coalgfg H$ since strong quotients of finitely generated
    objects are finitely generated. This diagram of the $T_i$ is a directed
    union with colimit $(X,x)$, both in $\B$ and in $\Coalg H$, so according
    to \autoref{unionoffgcoalgebras}, $(X,x)$ is lfg.
\end{proof}

\subsection*{Proof of \autoref{finalforfg}}
\begin{proof}
    The direction from left to right is clear, as $\Coalgfg \subseteq
    \Coalglfg$. For the other one, let $(S,s)$ be some lfg coalgebra. By
    \autoref{lfgdirectedunion}, it is the directed union of all its
    subcoalgebras with finitely generated carrier. For each subcoalgebra
    ${\inj_p}: (P,p) \rightarrow (S,s)$, there is a unique homomorphism
    $p^\dagger: (P,p)\to (L,\ell)$. By the uniqueness of $p^\dagger$ it
    follows that $L$ together with the $p^\dagger$ is a cocone. Hence there is a
    unique morphism $\exists! u: (S,s) \to (L,\ell)$ with $u\cdot \inj_p = p^\dagger$
    for each appropriate $(P,p)$. For any other morphism $\bar u: (S,s)\to
    (L,\ell)$ the equation $\bar u\cdot \inj_p = p^\dagger$ must hold as well,
    because $p^\dagger$ is unique. As the $\inj_p$ are jointly epic, one gets
    $\bar u = u$.
    \qed
\end{proof}

\subsection*{Proof of \autoref{lfgquotients}}
\begin{proof}
    Consider some strong quotient $q: (X,x) \to (Y,y)$ where $(X,x)$ is lfg. As
    $(X,x)$ is the directed colimit of its subcoalgebras with fg~carrier, we
    have that $(Y,y)$ -- the codomain of the strong epi-carried $q$ -- is the
    union of the images of these subcoalgebras by \autoref{unionsofimages}. The
    images themselves have a finitely generated carrier -- more precisely the
    factorization in $\Coalg H$ exist because $H$ preserves monos, by
    factorization. So $(Y,y)$ is the union of these images
    and thus is lfg.
    \qed
\end{proof}

\subsection*{Technical Lemmas for \autoref{LFFisIterative}}

The first task is to show that $\bar e$ is lfg. So essentially for each $f: S\to
X+\LFF H$ where $f$ is fg~we have to find a coalgebra through which $f$ factors, as
required by \autoref{lfgcoalgebra}. Roughly this is done in two steps:
firstly we construct the fg~image of $e$ in $\LFF H$, secondly the fg~image of $f$
in $\LFF H$, for the union $P$ of these images, we construct a coalgebra structure on
$X+P$ through which $f$ factors. In order to get this kind of image
factorization of $f$ and $e$ from the property of $X$ being finitely generated,
$\LFF H$ has to be expressed as a directed colimit of monos. This is done with the
following lemmas before going into the detail of the proof of the theorem.
\begin{lemma}
    Let $\Coalgfg'$ be the full subdiagram of $\Coalgfg$ consisting of those
    coalgebras $(A,a)$ where $a^\dagger: A \to \LFF H$ is a monomorphism. Then the
    forgetful functor $U': \Coalgfg' \to \C$ is a directed diagram of monos and
    filtered.
\end{lemma}
\begin{proof}
    At first, let us show that
    \begin{align}
        \text{for }(A,a)\text{ in \Coalgfg} \text{ there exists }(A',a')\text{ in
        $\Coalgfg'$ with }h: (A,a) \to  (A',a').
        \label{cofinal}
    \end{align}
    This follows directly from the (strong epi,mono) factorization which lifts
    from $\C$ to $\Coalgfg$. So
    $a^\dagger: A\to \LFF H$ factors into $h: A\twoheadrightarrow A'$ and
    $a'^\dagger: A' \rightarrowtail \LFF H$. The strong epi $h$ induces the structure
    $a': A' \to HA'$ and proves that both $h$ and $a'^\dagger$ are coalgebra
    homomorphisms. For the existence of upper bounds, which is required by the
    directedness, observe that coproducts exists in $\Coalgfg$, inducing upper
    bounds in $\Coalgfg'$ by \eqref{cofinal}.

    For any homomorphisms $g,h: (A_1,a_1) \to (A_2, a_2)$ we have $a_2^\dagger
    \cdot g = a_1^\dagger = a_2^\dagger\cdot h$. As $a_2^\dagger$ is monic,
    $g=h$, i.e. there is at most one arrow in each hom set of $\Coalgfg'$, which
    means that $U'$ is essentially small, a poset, and thus directed. As
    $a_1^\dagger$ is a mono, $h$ is a mono as well, so $U'$ is a directed
    diagram of monos.
\qed
\end{proof}
\begin{lemma}
    $\LFF H$ is the colimit of $U': \Coalgfg' H \to \C$.
\end{lemma}
\begin{proof}
    As \eqref{cofinal} proves, the inclusion functor $V: \Coalgfg' H \to \Coalgfg H$
    is a cofinal subdiagram. $\LFF H$ is the colimit of the forgetful functor $U:
    \Coalgfg H \to \C$, so $\colim U = \colim UV = \colim U'$.
\qed
\end{proof}

\subsection*{Proof of \autoref{LFFisIterative}}
\begin{proof}
  Let $e: X \to HX + \LFF H$ be an equation morphism with $X$ fg. In the
  following we prove that $e$ has a unique solution in $\LFF H$. The codomain
  $HX + \LFF H$ is the colimit of the following directed diagram of monos:
  \begin{itemize}
  \item The diagram scheme $\mathcal{D}$ is the product category containing
    pairs $(T \overset{t}\rightarrowtail HX, V \overset{v}{\rightarrow} HV)$
    consisting of an fg subobject of $HX$ and $(V,v) \in \Coalgfg' H$.
    $\mathcal{D}$ is directed, because both the fg subobjects of $HX$ and
    $\Coalgfg' H$ are.
    \vspace{1mm} % the fg subscript was too close to the \mathcal{D} of the
                   % next line

  \item The diagram $D: \mathcal{D}\to \C$ is defined by
    \[
      D(t,v) = \Im(t+ v^\dagger: T+V \to HX+\LFF H )
    \]
    on objects and by diagonalizaton on morphisms. By mono laws, all connecting
    morphisms are monic.

  \end{itemize}
  That $HX+\LFF H$ is indeed the colimit of $D$ follows from
  \autoref{unionsofimages} applied with $f=\id$. Because $X$ is fg, the morphism
  $e$ factors through one of the colimit injections, i.e. we obtain an $m:
  W\rightarrowtail HX+\LFF H$, $W$ fg, and $e$ such that $m\cdot e' = e$.
  Furthermore, choose some $t: T\rightarrowtail HX$ and $v: V\rightarrow HV$
  from $\mathcal{D}$ such that $W = D(t,v)$ as shown in the diagram below:
  
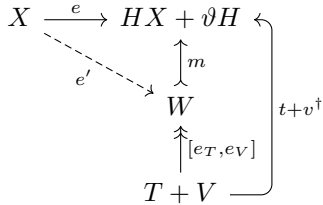
\[
    \begin{tikzcd}
      X
      \arrow{r}{e}
      \arrow[dashed]{dr}[swap]{e'}
      & HX + \LFF H
      \\
      & W
      \arrow[swap,>->]{u}{m}
      \\
      & T + V
      \arrow[swap,->>]{u}{[e_T,e_V]}
      \arrow[shiftarr={xshift=12mm}]{uu}[swap]{t+v^\dagger}
    \end{tikzcd}
  \]
  Since $T+V$ is fg, so is its strong quotient $W$.
  The intermediate object $W$ carries a coalgebra structure by diagonalization:
  
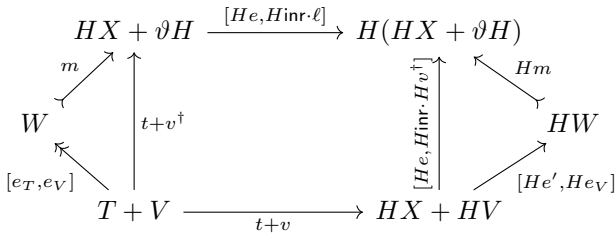
\[
    \begin{tikzcd}[column sep=0mm]
      &HX + \LFF H
      \arrow{r}{[He,H\inr\cdot \lff]}
      &[18mm] H(HX+\LFF H)
      \\
      W
      \arrow[>->]{ur}{m}
      &&& HW
      \arrow[>->,swap]{ul}{Hm}
      \\
      & T+V
      \arrow[->>]{ul}{[e_T,e_V]}
      \arrow{r}[swap]{t+v}
      \arrow{uu}[swap]{t+v^\dagger}
      & HX+HV
      \arrow{uu}[sloped,above]{[He,H\inr\cdot Hv^\dagger]}
      \arrow[swap]{ur}{[He',He_V]}
    \end{tikzcd}
  \]
  The inner square commutes on the left component trivially, and on the right
  component because $v^\dagger$ is a $H$-coalgebra homomorphism, and the two
  triangles by the previous diagram. This induces a
  morphism $w: W\to HW$ making $m$ and $e_V$ coalgebra homomorphisms. Since $m$
  is independent of the choice of $t$ and $v$ and since $Hm$
  is monic, $w$ is independent of the choice of $t$ and $v$.
  We have the following commuting diagram:
  \[
    \begin{tikzcd}[column sep=13mm, row sep=9mm]
      W
      \arrow[>->]{r}{m}
      & HX+\LFF H
      \arrow{r}{He'+\ell}
      \descto{d}{\begin{array}{c}v^\dagger~\text{coalgebra}\\[-1mm]\text{homomorphism}\end{array}}
      \descto[pos=0.7,xshift=5mm]{dr}{\begin{array}{c}
                             \text{finality}\\[-1mm]
                             \text{of~}\LFF H\end{array}}
      & |[xshift=10mm]| HW+H\LFF H
      \arrow[to path={
        ([xshift=-4mm]\tikztostart.south east) |- (\tikztotarget) \tikztonodes
      }]{dd}[]{[Hw^\dagger\!,\,H\LFF H]}
      \\[4mm]
      T+V
      \arrow[->>]{u}{[e_T,e_V]}
      \arrow[->>]{d}[swap]{[e_T,e_V]}
      \arrow{ur}[sloped,above]{t+v^\dagger}
      \arrow{r}{t+v}
      \descto{dr}{\text{Definition~of~}w}
      & HX+HV
      \arrow[bend left=8]{ur}[sloped,above]{He'+Hv^\dagger}
      \arrow{r}{He'+He_V}
      \arrow{d}[]{[He',He_V]}
      & HW+ HW
      \arrow{u}[sloped,above]{HW+Hw^\dagger}
      \arrow{d}[]{[Hw^\dagger,Hw^\dagger]}
      \\
      W
      \arrow{r}{w}
      \arrow{dr}[swap]{w^\dagger}
      & |[yshift=8mm]| HW
      \arrow{r}{H w^\dagger}
      \descto{d}{\begin{array}{c}w^\dagger~\text{coalgebra}\\[-1mm]\text{homomorphism}\end{array}}
      & |[yshift=8mm]| H\LFF H
      \arrow[bend left=10]{dl}{\ell^{-1}}
      \\[-6mm]
      & \LFF H
    \end{tikzcd}
  \]
  Since $[e_T,e_V]$ is an epimorphism, we therefore have
  \[
    w^\dagger
    = \ell^{-1}\cdot [Hw^\dagger \cdot He', \ell ]\cdot m
    = [\ell^{-1}\cdot Hw^\dagger \cdot He', \id_{\LFF H}]\cdot m
  \]
  and so $w^\dagger \cdot e'$ is a solution of $e$ in $(\LFF H, \ell^{-1})$:
  \[
    \begin{tikzcd}[column sep = 14mm,row sep=10mm]
      X
      \arrow{r}{e'}
      \arrow{d}[swap]{e}
      & W
      \arrow{dl}[swap]{m}
      \arrow{r}{w^\dagger}
      & \LFF H
      \\
      HX + \LFF H
      \arrow{r}{He' + \LFF H}
      & HW + \LFF H
      \arrow{r}{Hw^\dagger + \LFF H}
      & H\LFF H + \LFF H
      \arrow{u}[swap]{[\ell^{-1},\LFF H]}
    \end{tikzcd}
  \]
  To verify that this solution is unique, let $s: X\to \LFF H$ be any solution
  of $e$, i.e.~we have
  \begin{align}
    s = [\ell^{-1}\cdot Hs, \id_{\LFF H}]\cdot e.
    \label{anyOtherEquationSolution}
  \end{align}
  This defines a coalgebra homomorphism from $(W,w)$ to $\LFF H$:
  \[
    \begin{tikzcd}[column sep = 15mm,row sep= 14mm]
      W
      \arrow{r}{m}
      \arrow{d}[swap]{w}
      \descto[xshift=-5mm]{dr}{\begin{array}{c}
                    \text{Definition} \\[-1mm]
                    \text{of }w
                    \end{array}}
      & HX + \LFF H
      \arrow{r}{[\ell^{-1}\cdot Hs,\LFF H]}
      \arrow{d}[swap]{[He,H\inr\cdot \ell]}
      \arrow[xshift=3mm]{dr}[sloped,above]{[Hs,\ell]}
      \descto[xshift=-13mm]{dr}{\text{\eqref{anyOtherEquationSolution}}}
      &[14mm] \LFF H
      \arrow{d}{\ell}
      \\
      HW
      \arrow{r}{Hm}
      & H(HX+\LFF H)
      \arrow{r}{H[\ell^{-1}\cdot Hs,\LFF H]}
      & H \LFF H
    \end{tikzcd}
  \]
  Hence $[\ell^{-1}\cdot Hs,\id_{\LFF H}]\cdot m = w^\dagger$ and so
  \[
    w^\dagger \cdot e'
    = [\ell^{-1}\cdot Hs,\id_{\LFF H}]\cdot m \cdot e'
    = [\ell^{-1}\cdot Hs,\id_{\LFF H}]\cdot e,
    = s
  \]
  which completes the proof.
  \qed
\end{proof}

\subsection*{Technical Lemma for \autoref{alllfginitial}}
\begin{lemma}\label{fgiterativeunique}
    For an \fgiterative algebra $(A,\alpha: HA\to A)$ and a coalgebra $e: X\to HX$
    from $\Coalgfg$ there is a unique $\C$-morphism $u_e: X\to A$ such that $u_e
    = \alpha \cdot Hu_e \cdot e$.
    \[
        \begin{tikzcd}
            X \arrow[dashed]{r}{\exists! u_e} \arrow{d}[left]{e}
            & A \\
            HX \arrow{r}{Hu_e}
            \descto{ur}{\circlearrowleft}
            & HA \arrow{u}{\alpha}
        \end{tikzcd}
    \]
\end{lemma}
\begin{proof}
    Consider the equation morphism $\inl\cdot e: X\to HX+A$. For an arbitrary
    morphism $s: X\to A$, consider the following diagram:
    \[
        \begin{tikzcd}[column sep=1.6cm]
            X \arrow{rr}{s} \arrow{d}{e}
            &
            & A
            \\
            HX \arrow{r}{\inl}
            \arrow[ to path = |- (\tikztotarget) \tikztonodes
            ]{drr}[pos=0.75,below]{Hs}
            & HX+A \arrow{r}{Hs + A}
            &
            HA + A \arrow{u}[right]{[\alpha,A]}
            \\
            &{} \descto{u}{\circlearrowleft} &
            HA \arrow{u}[right]{\inl}
            \arrow[shiftarr={xshift=9ex}]{uu}[right]{\alpha}[left]{\circlearrowleft\ }
        \end{tikzcd}
    \]
    The lower part and the right-hand part always commute. But for the
    commutativity of the whole diagram consider the following sequence of
    equivalences:
    \begin{itemize}
        \item[] $s$ is a solution of $\inl\cdot e$ in $A$.
        \item[$\Leftrightarrow$] The upper square commutes.
        \item[$\Leftrightarrow$] $s = [\alpha,\id_A]\cdot \inl\cdot Hs\cdot e$
        \item[$\Leftrightarrow$] $s = \alpha\cdot Hs\cdot e$
    \end{itemize}
    So by the existence and the uniqueness of a solution of $\inl\cdot e$ in the
    \fgiterative algebra $A$, we get the desired morphism $u_e: X\to A$ with
    $u_e = \alpha \cdot Hu_e \cdot e$ and its uniqueness, by reading the
    equivalences from top or from bottom respectively.
    \qed
\end{proof}
\subsection*{Proof of \autoref{alllfginitial}}
\begin{proof}
    By \autoref{lfgdirectedunion}, $e: X\to HX$ is the union of the diagram $D$
    of its subcoalgebras $s: S\to HS$ with $S$ finitely
    generated. Denote the corresponding colimit injections by ${\inj_s: (S,s)\to
    (X,e)}$. Each such $s$ induces a unique morphism $u_s: S\to A$ with 
    \begin{equation}
    u_s = \alpha\cdot Hu_s\cdot s.
    \label{usproperty}
    \end{equation}
    For any coalgebra homomorphism $h: (R,r) \to (S,s)$ in $\Coalgfg$ the
    diagram
    \[
        \begin{tikzcd}
            R \arrow{r}{h} \arrow{d}[left]{r}
            & S \arrow{r}{u_s} \arrow{d}[left]{s} & A
            \\
            HR \arrow{r}[below]{Hh}
            & HS \arrow{r}[below]{Hu_s}
            & HA \arrow{u}[right]{\alpha}
        \end{tikzcd}
    \]
    commutes, because $h$ is a coalgebra homomorphism and because of the
    property of $u_s$. So $u_r = u_s\cdot h$. In other words, $A$ together with
    the morphisms $(u_s: S\to A)_{s: S\to HS\text{ lfg}}$ form a cocone for $D$
    in $\C$. This induces a unique morphism $u_e: X\to A$.

    For each $s: S\to HS$, $\inj_s: S\to X$ is a coalgebra homomorphism.
    Furthermore, we have is $u_s = u_e\cdot \inj_s$ in $\C$ by the universal property
    of $X$. So every part except possibly (ii) of the diagram
    \[
        \begin{tikzcd}
            S \arrow{r}{\inj_s} \arrow{d}[left]{s}
            \arrow[shiftarr={yshift=4ex}]{rr}{u_s}[below]{\circlearrowleft}
            \descto{dr}{\text{(i)}\circlearrowleft}
            & X \arrow{d}[left]{e} \arrow{r}{u_e}
            \descto{dr}{\text{(ii)}}
            & A
            \\
            HS \arrow{r}[below]{H\inj_s}
            \arrow[shiftarr={yshift=-4ex}]{rr}[below]{Hu_s}[above]{\circlearrowleft}
            & HX \arrow{r}[below]{Hu_e}
            & HA \arrow{u}[right]{\alpha}
        \end{tikzcd}
    \]
    commutes, as indicated. In particular the outer square square commutes which
    gives
    \[
    \alpha\cdot Hu_e\cdot e \cdot \inj_s = u_e\cdot \inj_s
    \text{ for every fg~subcoalgebra $(S,s)$ of $(X,e)$}.
    \]
    As the colimit injections $\inj_s$ are jointly epic, (ii) commutes.

    Conversely every $\C$-morphism $\tilde u_e: X\to A$ making (ii) commute, makes the
    bigger square (i)$+$(ii) commute and defines a family of morphisms $\tilde
    u_e\cdot \inj_s: S\to A$ having the property \eqref{usproperty} each.
    So by the uniqueness of the $u_s: S\to A$, we get $u_s = \tilde u_e\cdot
    \inj_s$. Using again that the $\inj_s$ are jointly epic, reduces the
    equation
    \[
        u_e\cdot \inj_s = u_s = \tilde u_e\cdot \inj_s
    \]
    to the desired uniqueness of $u_e$, namely $u_e = \tilde u_e$.
    \qed
\end{proof}

\subsection*{Proof of \autoref{lfpquotient}}
\begin{proof}
    Take a coalgebra $(X,x)$ with finitely generated carrier, which is the
    strong quotient of some fp~object $X'$ via $q: X' \twoheadrightarrow X$.
    By assumption, $X'$ is the strong quotient of a projective fp~object $X''$
    via $q': X'' \to X'$. As $H$ preserves strong epis, the projectivity of
    $X''$ induces the coalgebra structure $x''$:
    \[
        \begin{tikzcd}
            X'' \arrow[dashed]{r}{x''} \arrow[->>]{d}[left]{q'} &
            HX''\arrow[->>]{d}[right]{Hq'}
            \\
            X' \arrow[->>]{d}[left]{q} &
            HX' \arrow[->>]{d}[right]{Hq}
            \\
            X \arrow{r}{x} & HX
        \end{tikzcd}
    \]
    \qed
\end{proof}

\subsection*{Proof of \autoref{firstLFFImage}}
\begin{proof}
    First of all, $(\LFF H^T, \ell)$ is final for all $(TX,x^\sharp)$, with $X$
    finite, so it is a competing cocone for $(K,k)$:
    \[
        \begin{tikzcd}[ampersand replacement=\&]
        (TX,x^\sharp)
        \arrow{r}[above]{\inj_X}
        \arrow{dr}[below left]{Ux^{\sharp\dagger}}
        \&
        (K,k)
        \arrow[dashed]{d}[right]{w}
        \arrow[->>]{r}[above]{e}
        \arrow[to path={
                      -- ([yshift=2mm]\tikztostart.north)
                      -| ([xshift=4mm]\tikztotarget.east) \tikztonodes
                      -- (\tikztotarget.east)}
              ]{dr}[right,pos=0.65]{k^\dagger}
        \&
        (I,i)
        \arrow[>->]{d}[right]{m}
        \\
        {}
        \&
        (U\LFF H^T, U\ell)
        \arrow[>->]{r}{n}
        \&
        (\nu H, \tau)
        \end{tikzcd}
    \]
    Hence, $w$ is induced making the triangle commute. Any $(G,g)$ in $\Coalgfg
    H^T$ is the quotient of some $(TX,x^\sharp)$. And on the other hand, the
    $g^\dagger: (G,g) \to (\LFF H^T, \ell)$ are jointly epic. Hence, the
    $x^{\sharp\dagger}$ are jointly epic as well, and so the
    $Ux^{\sharp\dagger}$, too. Hence also $w$ is epic, and -- as we are in \Set\ --
    even a strong epimorphism. In other words, $(U\LFF H^T,U\ell)$ is the
    (unique) image of $(K,k)$ in $(\nu H,\tau)$.
    \qed
\end{proof}

\subsection*{Proof of \autoref{prop:LFFunion}}
\begin{proof}
    Combining the previous \autoref{firstLFFImage} together with the
    \autoref{unionsofimages} proves the first equality. For the second equality,
    consider any element $t \in TX$ and define a new coalgebra on $X+1$ by
    \[
    (Y,y) \equiv\big(
    \begin{tikzcd}[ampersand replacement = \&, column sep = 1.5cm]
        X+1 \rar{[x, x^\sharp(t)]}
        \& HTX \rar{HT\inl}
        \& HT(X+1)
    \end{tikzcd}
    \big).
    \]
    Clearly, $[\id_{TX}, t]: Y\to X$ is a $H^T$-coalgebra homomorphism, and
    $t \in y^{\sharp\dagger}\cdot\eta_Y^T[Y]$.
    \qed
\end{proof}

\subsection*{Definition of the Lifting of $S×(-)^\Sigma$ to $S$-algebras}
The $\Poly{-+\Sigma}$-algebra structure -- $S$-module structure,
monoid structure, $\Sigma$-pointing -- on $S×A^\Sigma$ can be defined using the
$\Poly{-+\Sigma}$-structure on $A$ as follows:
\[
\begin{nicearray}{l@{\hskip 3mm}l@{\hskip 3mm}l@{\hskip 3mm}l}
        \text{Structure} & \text{Connective} & \text{in }S &\text{in }A^\Sigma\\
        \spmidrule
        \text{$S$-Module}
        & 0 & 0_S& a\mapsto 0_A \\
        & (o_1,\delta_1) + (o_2, \delta_2)
        & o_1 + o_2
        & a\mapsto \delta_1(a) + \delta_2(a)
        \\
        & s\cdot (o_1,\delta_1)
        & s\cdot o_1
        & a\mapsto s\cdot \delta_1(a)
        \\
        \spmidrule
        \text{Monoid}
        & 1 & 1_S& a\mapsto 0_A \\
        & (o_1,\delta_1) * (o_2, \delta_2)
        & o_1 \cdot o_2
        & a\mapsto
            \delta_1(a)\cdot \fuse{o_2,\delta_2}
            + i(o_1)\cdot \delta_2(a) \\
        \spmidrule
        \text{$\Sigma$-pointing} & b\in \Sigma & 0_S& b\mapsto 1_A,\quad a\mapsto 0_A, b\neq a
\end{nicearray}
\]
The defined connectives only makes use of connectives from $S$ (seen as a
$S$-algebra) and from the $S$-algebra $A$, so $H$ maps any
$\Poly{-+\Sigma}$-algebra homomorphism $h: A\to B$ to again a homomorphism $Hh:
S×A^\Sigma\to S×B^\Sigma$. In total, we have a lifting $H^T:
\Set^T\to\Set^T$ of $H$, as soon as we have checked the $S$-algebra axioms for
$HA = S \times A^\Sigma$.
\subsection*{$S$-algebra connective preserving $[-]: S×A^\Sigma \to A$}
In order to show that $S×A^\Sigma$ is indeed an $S$-algebra, it comes handy to
establish some identities for $[-]: S×A^\Sigma \to A$ first, namely, that it
preserves the proposed $S$-algebra structure in the expected manner.
It preserves the $S$-module connectives zero
\begin{align*}
    [0_S, a\mapsto 0_A]
    = i(0_S) + \sum_{b\in\Sigma} j(b)\cdot 0_A
    = 0_A
\end{align*}
addition,
\begin{align*}
    &
    [o_1+o_2, a\mapsto \delta_1(a) +\delta_2(a)]
    \\ =\ &
    i(o_1+o_2) + \sum_{b\in \Sigma} \big(j(b)\cdot (\delta_1(a) + \delta_2(a))\big)
    \\ =\ &
    i(o_1)+i(o_2) + \sum_{b\in \Sigma} \big(j(b)\cdot \delta_1(a)\big)
    + \sum_{b\in \Sigma}\big(j(b)\cdot \delta_2(a))\big)
    \\ =\ &
    [o_1, \delta_1] + [o_2, \delta_2]
\end{align*}
and scalar multiplication:
\begin{align*}
    [s\cdot o,a\mapsto s\cdot \delta(a)]
    &=
    i(s\cdot o) + \sum_{b\in \Sigma}\big(j(b)\cdot (s\cdot \delta(b))\big)
    =
    s\cdot i(o) + \sum_{b\in \Sigma}s\cdot \big(j(b)\cdot \delta(b)\big)
    \\
    &=
    s\cdot \left(i(o) + \sum_{b\in \Sigma}\big(j(b)\cdot \delta(b)\big)\right)
    =
    s\cdot [o,\delta]
\end{align*}
The monoid connectives are preserved as well:
\[
    [1_S, a\mapsto 0_A]
    = i(1_S) +\sum_{b\in\Sigma}j(b)\cdot 0_A = i(1_S) = 1_A
\]
\begin{align*}
    [o_1,\delta_1]\cdot [o_2,\delta_2]
    & = \left(i(o_1) + \sum_{b\in \Sigma} j(b)\cdot \delta_1(b)\right)\cdot [o_2,\delta_2]
    \\ &
    = i(o_1)\cdot [o_2,\delta_2] + \sum_{b\in\Sigma}
    j(b)\cdot\delta_1(b)\cdot [o_2,\delta_2]
    \\ &
    = i(o_1)\cdot \left(i(o_2)+\sum_{b\in\Sigma}j(b)\cdot \delta_2(b)\right)
    + \sum_{b\in\Sigma} j(b)\cdot\delta_1(b)\cdot [o_2,\delta_2]
    \\ &
    = i(o_1\cdot o_2)+\sum_{b\in\Sigma}j(b)\cdot i(o_1)\cdot \delta_2(b)
    + \sum_{b\in\Sigma} j(b)\cdot\delta_1(b)\cdot [o_2,\delta_2]
    \\ &
    = i(o_1\cdot o_2)+\sum_{b\in\Sigma}j(b)\cdot \big(i(o_1)\cdot \delta_2(b)
    + \delta_1(b)\cdot [o_2,\delta_2]\big)
    \\ &
    = \big[
        o_1\cdot o_2, a\mapsto i(o_1)\cdot\delta_2(a)+\delta_1(a)\cdot[o_2,\delta_2]
    \big]
    = \big[(o_1,\delta_1)* (o_2,\delta_2)\big]
\end{align*}

\subsection*{S-Algebra axioms}
Firstly, note that $(HA,0,+,\cdot)$ fulfills all the $S$-Module axioms, because
$0,+,\cdot$ are defined point-wise in $A$. Secondly, $(HA,1,*)$ is a monoid:
\begin{align*}
    (1_S,a\mapsto 0_A)* (o,\delta)
    &= \big(1_S\cdot o, a\mapsto 0_A\cdot [o,\delta] + i(1_S)\cdot \delta(a)\big)
    = \big(o, a\mapsto \delta(a)\big)
\\
    (o,\delta)* (1_S,a\mapsto 0_A)
    &= \big(o\cdot 1_S, a\mapsto \delta(a)\cdot [1_S,a\mapsto 0_A] + i(o)\cdot 0_A\big)
    \\
    &=  \big(o,  a\mapsto \delta(a)\cdot 1_A + 0_A\big)
    =  (o,\delta)
\end{align*}
\begin{align*}
    &\big((o_1,\delta_1)* (o_2,\delta_2)\big)*(o_3,\delta_3)
    \\ =\ &
    \big(o_1\cdot o_2,
        a\mapsto \delta_1(a)\cdot [o_2,\delta_2]+ i(o_1)\cdot \delta_2(a)\big)
            * (o_3,\delta_3)
    \\ =\ &
    \big(o_1\cdot o_2\cdot o_3, a\mapsto
        \big(\delta_1(a)\cdot [o_2,\delta_2]
            + i(o_1)\cdot \delta_2(a)\big)
             \cdot [o_3,\delta_3]
          + i(o_1\cdot o_2)\cdot \delta_3(a)\big)
    \\ =\ &
    \big(o_1\cdot o_2\cdot o_3, a\mapsto
        \delta_1(a)\cdot [o_2,\delta_2]\cdot [o_3,\delta_3]
            + i(o_1)\cdot \delta_2(a)\cdot [o_3,\delta_3]
          + i(o_1\cdot o_2)\cdot \delta_3(a)\big)
    \\ =\ &
    \big(o_1\cdot o_2\cdot o_3, a\mapsto
        \delta_1(a)\cdot \big[(o_2,\delta_2)*(o_3,\delta_3)\big]
            + i(o_1)\cdot \big(\delta_2(a)\cdot [o_3,\delta_3]
          + i(o_2)\cdot \delta_3(a)\big)\big)
    \\ =\ &
    (o_1,\delta_1)* \big(o_2\cdot o_3, a\mapsto \delta_2(a)\cdot [o_3,\delta_3] +
    i(o_2)\cdot \delta_3(a)\big)
    \\ =\ &
    (o_1,\delta_1)* \big((o_2,\delta_2)*(o_3,\delta_3)\big)
\end{align*}
What remains is the bilinearity of $*$ with respect to the
$S$-Module structure. For bilinearity of $*$ in the first argument, we use the
very same properties in $A$:
\begin{align*}
    &
    \big((o_1,\delta_1) + (o_2,\delta_2)\big) * (o_3,\delta_3)
    \\ =\ &
    (o_1+o_2, a\mapsto \delta_1(a) + \delta_2(a)) * (o_3,\delta_3)
    \\ =\ &
    \big((o_1+o_2)\cdot o_3, a\mapsto
        (\delta_1(a)+\delta_2(a))\cdot [o_3,\delta_3] + i(o_1+o_2)\cdot
        \delta_3(a)\big)
    \\ =\ &
    \big(o_1\cdot o_3 +o_2\cdot o_3, a\mapsto
        \delta_1(a)\cdot [o_3,\delta_3]+\delta_2(a)\cdot [o_3,\delta_3] +
        (i(o_1)+i(o_2))\cdot \delta_3(a)\big)
    \\ =\ &
    (o_1,\delta_1) * (o_3,\delta_3) + (o_2,\delta_2) * (o_3,\delta_3)
\end{align*}
\begin{align*}
    \big(s\cdot (o_1,\delta_1)\big) * (o_2,\delta_2)
    &= \big(s\cdot o_1 \cdot o_2,
        a\mapsto s\cdot \delta_1(a)\cdot \fuse{o_2,\delta_2}
            + i(s\cdot o_1)\cdot \delta_2(a)\big)
    \\
    &= \big(s\cdot o_1 \cdot o_2,
        a\mapsto s\cdot (\delta_1(a)\cdot \fuse{o_2,\delta_2}
            + i(o_1)\cdot \delta_2(a))\big)
    \\
    &= s\cdot \big((o_1,\delta_1)* (o_2,\delta_2)\big)
\end{align*}
\begin{align*}
    (0_S, a\mapsto 0_A)*(o,\delta)
    &= \big(0_S\cdot o, a\mapsto 0_A\cdot [o,\delta] + i(0_S)\cdot \delta(a)\big)
    \\
    &= \big(0_S\cdot o, a\mapsto 0_A\cdot [o,\delta] + 0_A\cdot \delta(a)\big)
    = \big(0_S, a\mapsto 0_A\big)
\end{align*}
%The preservation of $0$ and scalar multiplication are analogous.
Finally, linearity in the second argument of $*$ using the identities for $[-]$:
\begin{align*}
    &
    (o_1,\delta_1) * \big((o_2,\delta_2) + (o_3,\delta_3)\big)
    \\ =\ &
    (o_1\cdot (o_2+o_3), a\mapsto \delta_1(a)\cdot [o_2+o_3, a\mapsto
    \delta_2(a)+\delta_3(a)]+ i(o_1)\cdot (\delta_2(a) + \delta_3(a))
    \\ =\ &
    (o_1\cdot o_2+ o_1\cdot o_3, a\mapsto \delta_1(a)\cdot [o_2,\delta_2]
    + \delta_1(a)\cdot [o_3,\delta_3]
    + i(o_1)\cdot \delta_2(a)
    + i(o_1)\cdot  \delta_3(a)
    \\ =\ &
    (o_1,\delta_1) * (o_2,\delta_2) +(o_1,\delta_1) * (o_3,\delta_3)
\end{align*}
\begin{align*}
    &
    (o_1,\delta_1) * \big(s\cdot (o_2,\delta_2)\big)
    = (o_1,\delta_1) * (s\cdot o_2,a\mapsto s\cdot \delta_2)\big)
    \\ =\ &
    \big(o_1\cdot (s\cdot o_2), \delta_1(a)\cdot [s\cdot o_2,a\mapsto s\cdot
    \delta_2(a)] + i(o_1)\cdot (s\cdot \delta_2(a))\big)
    \\ =\ &
    \big(o_1\cdot (s\cdot o_2), \delta_1(a)\cdot (s\cdot [o_2,\delta_2]) +
    i(o_1)\cdot (s\cdot \delta_2(a))\big)
    \\ =\ &
    \big(s\cdot (o_1\cdot o_2), s\cdot (\delta_1(a)\cdot [o_2,\delta_2]) +
    s\cdot (i(o_1)\cdot \delta_2(a))\big)
    = s\cdot \big((o_1,\delta_1) * (o_2,\delta_2)\big)
\end{align*}
\begin{align*}
    (o,\delta) * (0_S, a\mapsto 0_A)
    &= \big(o\cdot 0_S,a\mapsto  \delta(a) \cdot [0_S, a\mapsto 0_A] + i(o)\cdot 0_A\big)
    \\ &=
    \big(o\cdot 0_S, a\mapsto \delta(a) \cdot 0_A + 0_A\big)
    = (0_S,a\mapsto 0_A)
\end{align*}
So for any $S$-algebra $A$, $S×A^\Sigma$ is an $S$-algebras too and hence $[-]:
S×A^\Sigma \to A$ an  $S$-algebra homomorphism.

\subsection*{Proof of \autoref{bisimilarAlgebra}}
\begin{proof}
    In other words, let us prove that $R = \{(\fuse{c(w)},w)\mid w \in A\}$ is a
    bisimulation. First, take $c = \fpair{o,\delta}$ in \Set (not in $\Set^T$)
    and note that the following holds for any $b\in \Sigma$ and $v \in A$ (where $\varrho_b: \Sigma \to A$ with $\varrho_b(b) = 1$ and $\varrho_b(a) =0$ for $a \neq b$):
    \begin{align*}
        c(j(b)) * c(v)
        &= (0_S,\varrho_b) * c(v)
        = (0_S,\varrho_b) * \big(o(v), \delta(v)\big)
        \\ &
        = \big(0_S\cdot o(v),
            a\mapsto \varrho_b(a)\cdot \fuse{c(v)}
                     + i(0_S)\cdot \delta(v)(a) \big)
        \\ &
        = \big(0_S, a\mapsto \varrho_b(a) \cdot \fuse{c(v)}\big)
    \end{align*}
    The following shows that $R$ is a bisimulation:
    \begin{align*}
        c(\fuse{c(w)})
        &= c(\fuse{o(w), \delta(w)})
        = c\left(i(o(w)) + \sum_{b\in \Sigma} \big(j(b)\cdot \delta(w)(b)\big)\right)
        \\ &
        = c\big(i(\underbrace{o(w)}_{\in S})\big) + \sum_{b\in \Sigma} c\big(j(b)\big)* c\big(\delta(w)(b)\big)
        \\ &
        = (o(w), a\mapsto 0_A) +
          \sum_{b\in \Sigma} \big(0_S, a\mapsto \varrho_b(a) \cdot \fuse{c(\delta(w)(b))}\big)
        \\ &
        = (o(w), a\mapsto 0_A) +
          \left(0_S, a\mapsto \sum_{b\in \Sigma} \varrho_b(a) \cdot
          \fuse{c(\delta(w)(b))}\right)
        \\ &
        = (o(w), a\mapsto 0_A) +
          \left(0_S, a\mapsto \fuse{c(\delta(w)(a))}\right)
      %  \\ &
      %  = (o(w), a\mapsto 0_A) +
      %    \left(0_S, a\mapsto \fuse{c(\delta(w)(a))}\right)
        \\ &
        = (o(w), a\mapsto \fuse{c(\delta(w)(a))})
    \end{align*}
    This says that for any $v\in A$, $o(\fuse{c(v)}) = o(v)$ and for all $a\in \Sigma$
    \[
        \delta(\fuse{c(v)})(a)
        = \fuse{c(\delta(v)(a))}
        \ R\  \delta(v)(a).
    \]
    i.e.~$R$ is a bisimulation.
    \qed
\end{proof}

\subsection*{Proof of \autoref{CFSameCoalgebra}}
\begin{proof}[By induction on $u$ w.r.t.~the connectives of $S$-algebras]
    Put $c^\sharp = \fpair{o^\sharp,\delta^\sharp}$.
    \begin{itemize}
    \item
    \emph{Base Case:} For any $x\in X$, $x\in \Poly{X}$ and $x \in \Poly{X+\Sigma}$ are
    behaviourally equivalent by construction of $\hat c$ and $c^\sharp$.

    \item
    \emph{Step \textqt{$S$-Module-Structure}:}
    The definition of $\hat c$ on $S$-Module connectives is point-wise
    \cite[Sect.~3]{jcssContextFree}, and thus identical to the definition of
    $c^\sharp$.

    \item
    \emph{Step \textqt{Monoid-Structure}:}
    The neutral element is mapped by $\hat c$ to $(1,a\mapsto 0)$
    \cite[Sect.~4]{jcssContextFree}, and this is identical to the definition $c^\sharp$.

    For polynomials $v,w \in \Poly{X}$ and $v',w'\in \Poly{{X+\Sigma}}$, assume
    that $v\sim v'$, $w\sim w'$ (with $\sim$ denoting behavioural equivalence). We have
    \[
        \hat o(v*w)
        \overset{\eqref{WEAproperty}}{=} \hat o(v)\cdot \hat o(w)
        \overset{\text{IH}}{=} o^\sharp(v') \cdot o^\sharp(w')
        \overset{\eqref{Hmultiplication}}{=} o^\sharp(v'* w').
    \]
    Note that final homomorphism $\hat c^\dagger: \Poly{X}
    \to \nu H$ preserves multiplication by \cite[Prop 15]{jcssContextFree} and
    the final $c^{\sharp\dagger}$ as well, because it lives in $\Set^T$. So for
    any $x \sim x'$ and $y \sim y'$, $x,y,\in \Poly{X}$, $x',y'\in TX$, we have:
    \[
        \hat c^\dagger(x* y)
        = \hat c^\dagger(x) * \hat c^\dagger(y)
        \overset{x\sim x'}{\underset{y\sim y'}{=}} c^{\sharp\dagger}(x') * c^{\sharp\dagger}(y')
        = c^{\sharp\dagger}(x' * y'),
    \]
    i.e.~$\sim$ is a congruence for $*$ (and also for $+$). The hypothesis $v\sim v'$ implies $\hat\delta(v,a) \sim \delta^\sharp(v',a)$.
    For $a\in \Sigma$, 
    \begin{align*}
        \hat\delta(v*w, a)
        &\overset{\eqref{WEAproperty}}{=} \hat\delta(v,a)*w+\hat o(v)*\hat\delta(w,a)
        \\ &
        \overset{\text{IH}}{\sim} \delta^\sharp(v',a)*w'+o^\sharp(v)*\delta^\sharp(w',a)
        \\ &
        \overset{\mathclap{\text{\autoref{bisimilarAlgebra}}}}{\sim}\,\,\quad
            \delta^\sharp(v',a)*\fuse{o^\sharp(w'),\delta^\sharp(w')}
            +o^\sharp(v)*\delta^\sharp(w',a)
        \\ &
        \overset{\eqref{Hmultiplication}}{=}
            \delta^\sharp(v'*w', a).
    \end{align*}
    So $v*w \sim v' * w'$.
    \qed
    \end{itemize}
\end{proof}

\subsection*{Proof of \autoref{mndfilteredcolimits}}
\begin{proof}
    Let $D: \D\to \Mndf(\C)$, $Di = (M_i, \eta^i, \mu^i)$ be a filtered diagram.
    Take its colimit $M = \colim D$ with injections $\inj_i: M_i \to M$ in
    $\Funf(\C)$ and define a monad unit by
    \[
        \eta \equiv \big(
            \Id \xrightarrow{\eta^i} M_i
                \xrightarrow{\inj_i} M
        \big),
        \quad\text{ for any }i\in \D.
    \]
    Similarly, define the monad multiplication $\mu: MM\to M$ as the unique
    natural transformation with
    \[
        \begin{tikzcd}[ampersand replacement=\&]
            M_iM_i
                \arrow{r}{\mu^i}
                \arrow{d}[left]{\inj_i* \inj_i}
            \& M_i
                \arrow{d}[right]{\inj_i}
            \\
            MM
                \arrow[dashed]{r}{\mu}
            \& M
        \end{tikzcd}
        \quad\text{for any }i\in \D.
    \]
    The filteredness of $D$ proves the independence of the choice of $i$: for
    any other candidate $j\in \D$ choose an upper bound $m_{i,k}: M_i
    \rightarrow M_k \leftarrow M_j: m_{j,k}$ of $M_i$ and
    $M_j$. Then we have a commutative diagram
    
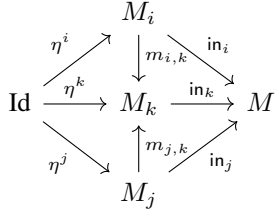
\[
    \begin{tikzcd}[ampersand replacement=\&]
        \&
        M_i
            \arrow{dr}[above right]{\inj_i}
            \arrow{d}{m_{i,k}}
        \\
        \Id
        \arrow{ur}[above left]{\eta^i}
        \arrow{dr}[below left]{\eta^j}
        \arrow{r}[above]{\eta^k}
        \& M_k
            \arrow{r}[above]{\inj_k}
        \& M
        \\
        \&
        M_j
            \arrow{u}[right]{m_{j,k}}
            \arrow{ur}[below right]{\inj_j}
    \end{tikzcd}
    \]
    The left-hand triangles commute because $m_{i,k}, m_{j,k}$ are monad
    morphisms and the right-hand triangles because $m_{i,k}, m_{j,k}$ are
    connecting natural transformations of $D$ and the $\inj$ the colimit
    injections.

    Note that $(M_iM_i)_{i\in \D}$ is a filtered diagram with colimit $MM$ in
    $\Funf(\C)$. Let us check the monad laws:
    \begin{itemize}
    \item Unit laws: the diagrams
    
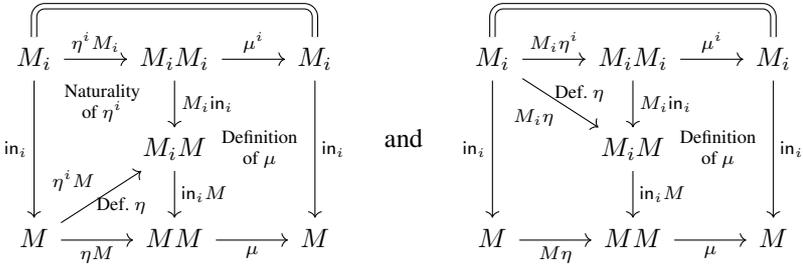
\[
        \begin{tikzcd}[ampersand replacement=\&]
            M_i
            \arrow{dd}[left]{\inj_i}
            \arrow{r}{\eta^iM_i}
            \arrow[oldequal, shiftarr={yshift=7mm}]{rr}
        \descto[fill=none]{dr}{\text{\parbox{2cm}{\centering\scriptsize
                            Naturality\\of $\eta^i$ }}}
            \&
            M_iM_i
            \arrow{d}{M_i\inj_i}
            \arrow{r}{\mu^i}
        \descto[fill=none,xshift=2mm]{ddr}{\text{\parbox{2cm}{\centering\scriptsize
                            Definition\\of $\mu$ }}}
            \&M_i
            \arrow{dd}{\inj_i}
            \\
            {}
        \descto[fill=none,pos=0.8]{dr}{\text{\parbox{2cm}{\centering\scriptsize
                            Def.~$\eta$ }}}
            \&M_iM
                \arrow{d}{\inj_iM}
            \\
            M
                \arrow{ur}[above left]{\eta^iM}
                \arrow{r}[below]{\eta M}
            \& MM \arrow{r}[below]{\mu}
            \& M
        \end{tikzcd}
        \quad\text{ and }\quad
        \begin{tikzcd}[ampersand replacement=\&]
            M_i
                \arrow{r}{M_i\eta^i}
                \arrow[oldequal, shiftarr={yshift=7mm}]{rr}
                \arrow{dr}[below left]{M_i\eta}
                \arrow{dd}[left]{\inj_i}
            \& M_iM_i
                \arrow{r}{\mu^i}
                \arrow{d}{M_i\inj_i}
        \descto[fill=none,xshift=2mm]{ddr}{\text{\parbox{2cm}{\centering\scriptsize
                            Definition\\of $\mu$ }}}
            \& M_i
                \arrow{dd}{\inj_i}
            \\
            {}
        \descto[fill=none,pos=0.8]{ur}{\text{\parbox{2cm}{\centering\scriptsize
                            Def.~$\eta$ }}}
            \& M_iM
                \arrow{d}{\inj_iM}
            \\
            M
                \arrow{r}[below]{M\eta}
            \& MM
                \arrow{r}[below]{\mu}
            \& M
        \end{tikzcd}
    \]
    commute. As the $\inj_i$ are jointly epic, $(M,\eta,\mu)$ fulfills the unit
    laws.

    \item Associativity:
    
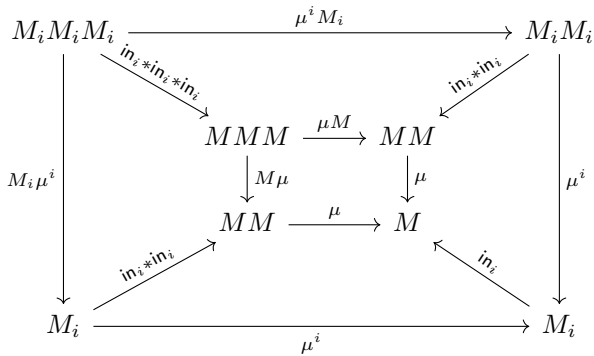
\[
    \begin{tikzcd}[ampersand replacement=\&]
        M_iM_iM_i
            \arrow{rrr}{\mu^iM_i}
            \arrow{ddd}[left]{M_i\mu^i}
            \arrow{dr}[sloped,above]{\inj_i*\inj_i*\inj_i}
        \&\&\&
        M_iM_i
            \arrow{ddd}{\mu^i}
            \arrow{dl}[sloped,above]{\inj_i*\inj_i}
        \\[2mm]
        {}
        \& MMM \arrow{r}{\mu M}
              \arrow{d}{M \mu}
        \& MM \arrow{d}{\mu}
        \\
        {}
        \& MM \arrow{r}{\mu}
        \& M
        \\[2mm]
        M_i \arrow{rrr}[below]{\mu^i}
            \arrow{ur}[sloped,above]{\inj_i*\inj_i}
        \&\&\&
        M_i
            \arrow{ul}[sloped,above]{\inj_i}
    \end{tikzcd}
    \]
    The outside commutes, and by definition of $\mu$ also all inner parts
    (except possibly for the middle square). As the $\inj_i*\inj_i*\inj_i$ are
    jointly epic, the middle square commutes as well. 
    \end{itemize}
    By definition of $\eta$ and $\mu$, each $\inj_i: M_i\to M$ is a monad
    morphism. In fact, $\eta$ and $\mu$ are the unique natural transformations
    making the diagrams (in the definition) commute, i.e.~are the unique monad
    structure on $M$ such that $\inj_i$ is a monad morphism.

    To see that $(M,\eta,\mu)$ is a colimiting cocone, consider another cocone
    $n_i: M_i \to N$ in $\Mndf(\C)$. This induced a unique natural
    transformation $m: M\to N$ with $n_i = m\cdot \inj_i$. To see that $m$ is
    also a monad morphism, use the jointly epicness of the $\inj_i$:
    \[
        m \cdot \eta = m\cdot \inj_i \cdot \eta^i = n_i\cdot \eta^i = \eta^N,
    \]

    Consider the following diagram:
    
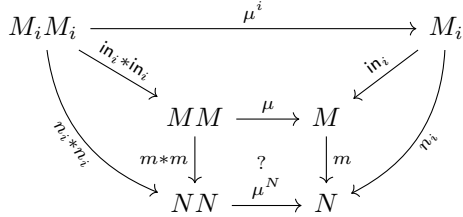
\[
    \begin{tikzcd}[ampersand replacement=\&]
        M_iM_i
            \arrow{rrr}{\mu^i}
            \arrow{dr}[sloped,above]{\inj_i*\inj_i}
            \arrow[bend right]{ddr}[sloped,below]{n_i*n_i}
        \&\&\& M_i
            \arrow{dl}[sloped,above]{\inj_i}
            \arrow[bend left]{ddl}[sloped,below]{n_i}
        \\
        {}
        \& MM \arrow{r}{\mu}
             \arrow{d}[left]{m*m}
             \descto{dr}{?}
        \& M \arrow{d}{m}
        \\
        {}
        \& NN \arrow{r}{\mu^N}
        \& N
    \end{tikzcd}
    \]
    The outside commutes, because $n_i$ is a monad morphism. The outer triangles
    commute on the level of $\Funf(\C)$ and the upper part commutes because
    $\inj_i$ is a monad morphism. Again, as the $\inj_i*\inj_i$ are jointly
    epic, the inner square commutes as well, hence $m$ is a monad morphism.
    \qed
\end{proof}

\subsection*{Proof of \autoref{quasiEpiPreservation}}
\begin{proof}
    Strong epis in slice categories are carried by strong epis, so consider a
    strong epi $q: A\to B$ in $\Mndf(\Set)$.
    Consider the (strong epi,mono)-factorizations of the components in $\Funf$:
    \[
    \begin{tikzcd}
    M \arrow[shiftarr={yshift=7mm}]{rr}{q}
       \arrow[->>]{r}{e}
    & I \arrow[>->]{r}{m}
    & N
    \end{tikzcd}
    \]
    The factorization lifts further to $\Mndf(\Set)$, i.e.~we have factorized
    the monad morphism $q$ into an epi $e$ and  a mono $m$ in $\Mndf(\Set)$. As
    any strong epi is also extremal, we get that $m$ is an isomorphism. Hence
    $q$ has epic components.
    All
    \Set-functors preserve (strong) epis, so $Hq_X+\Id$ is epic for any set $X$ and so
    the natural transformation $Hq+\Id$ as well.
\end{proof}

\subsection*{Proof of \autoref{EQquotient}}
\begin{proof}
    $(B,\beta)$ is the strong quotient of a $(F^{H+V}, \hat\kappa
    \cdot \inl)$, which again is a quotient of $(F^{H+P}, \hat\kappa\cdot
    \inl)$, where $P$ a polynomial functor and therefore an epi-projective in $\Funf(\C)$.
    \[
        \begin{tikzcd}
            (F^{H+P}, \hat\kappa\cdot \inl)
            \arrow[->>]{r}{q_P}
            \arrow[->>,shiftarr={yshift=-6mm}]{rr}[below]{=: q}
            &
            (F^{H+V}, \hat\kappa\cdot \inl)
            \arrow[->>]{r}{q_V}
            &
            (B, \beta)
            \arrow[->]{r}{b}
            &
            \H (B, \beta)
        \end{tikzcd}
    \]
    This corresponds to a natural transformation $\overline{b\cdot q}: P \to
    HB+\Id$. As $P$ is projective and by \autoref{quasiEpiPreservation} $\Hf q$ is epic as a
    natural transformation, we get a natural transformation $p: P \to
    HF^{H+P}+\Id$ such that the diagram on the left below commutes:
    \[
        \begin{tikzcd}
            P \arrow{dr}[below left]{\overline{b\cdot q}}
            \arrow[dashed]{r}{p}
            &
            HF^{H+P}+\Id
            \arrow[->>]{d}{Hq+\Id}
            \\
            {}
            &
            HB+ \Id
        \end{tikzcd}
        \quad\Longleftrightarrow
        \begin{tikzcd}
            (F^{H+P},\hat\kappa\cdot \inl)
            \arrow{r}{\bar p}
            \arrow[->>]{d}[left]{q}
            &
            \Hf(F^{H+P},\hat\kappa\cdot \inl)
            \arrow{d}{Hq+\Id}
            \\
            (B,\beta)
            \arrow{r}{b}
            &
            \Hf(B,\beta)
        \end{tikzcd}
    \]
    It follows that the coalgebra $b$ is the strong quotient of the coalgebra $\bar p$, which is a coalgebra in $\EQ$.
\end{proof}

    \end{appendix}
\else
    % no appendix in non-full version
\fi

\end{document}